\documentclass[a4paper,english,british,reqno]{article}
\usepackage[T1]{fontenc}
\usepackage[latin9]{inputenc}
\usepackage{color}
\usepackage{babel}
\usepackage{mathrsfs}
\usepackage{amsthm}
\usepackage{amsmath}
\usepackage{amssymb}
\usepackage{setspace}
\usepackage{esint}
\usepackage[unicode=true,pdfusetitle,
 bookmarks=true,bookmarksnumbered=false,bookmarksopen=false,
 breaklinks=false,pdfborder={0 0 0},backref=false,colorlinks=true]
 {hyperref}
\usepackage{breakurl}

\makeatletter

\numberwithin{equation}{section}
\numberwithin{figure}{section}
\theoremstyle{plain}
\newtheorem{thm}{\protect\theoremname}
  \theoremstyle{definition}
  \newtheorem{defn}[thm]{\protect\definitionname}
  \theoremstyle{definition}
  \newtheorem{example}[thm]{\protect\examplename}
  \theoremstyle{remark}
  \newtheorem{rem}[thm]{\protect\remarkname}
  \theoremstyle{plain}
  \newtheorem{lem}[thm]{\protect\lemmaname}
  \theoremstyle{plain}
  \newtheorem{prop}[thm]{\protect\propositionname}
  \theoremstyle{plain}
  \newtheorem{cor}[thm]{\protect\corollaryname}

\usepackage{babel}

\usepackage{amsfonts}
\usepackage{latexsym}

  \providecommand{\corollaryname}{Corollary}
  \providecommand{\definitionname}{Definition}
  \providecommand{\examplename}{Example}
  \providecommand{\lemmaname}{Lemma}
  \providecommand{\propositionname}{Proposition}
  \providecommand{\remarkname}{Remark}
\providecommand{\theoremname}{Theorem}

\makeatother

  \addto\captionsbritish{\renewcommand{\corollaryname}{Corollary}}
  \addto\captionsbritish{\renewcommand{\definitionname}{Definition}}
  \addto\captionsbritish{\renewcommand{\examplename}{Example}}
  \addto\captionsbritish{\renewcommand{\lemmaname}{Lemma}}
  \addto\captionsbritish{\renewcommand{\propositionname}{Proposition}}
  \addto\captionsbritish{\renewcommand{\remarkname}{Remark}}
  \addto\captionsbritish{\renewcommand{\theoremname}{Theorem}}
  \addto\captionsenglish{\renewcommand{\corollaryname}{Corollary}}
  \addto\captionsenglish{\renewcommand{\definitionname}{Definition}}
  \addto\captionsenglish{\renewcommand{\examplename}{Example}}
  \addto\captionsenglish{\renewcommand{\lemmaname}{Lemma}}
  \addto\captionsenglish{\renewcommand{\propositionname}{Proposition}}
  \addto\captionsenglish{\renewcommand{\remarkname}{Remark}}
  \addto\captionsenglish{\renewcommand{\theoremname}{Theorem}}
  \providecommand{\corollaryname}{Corollary}
  \providecommand{\definitionname}{Definition}
  \providecommand{\examplename}{Example}
  \providecommand{\lemmaname}{Lemma}
  \providecommand{\propositionname}{Proposition}
  \providecommand{\remarkname}{Remark}
\providecommand{\theoremname}{Theorem}

\begin{document}

\title{Generalized solutions in PDE's and the Burgers' equation}

\author{Vieri Benci\thanks{\textsc{V. Benci, Dipartimento di Matematica, Università degli Studi
di Pisa, Via F. Buonarroti 1/c, 56127 Pisa, ITALY and Centro Linceo
Interdisciplinare Beniamino Segre, Palazzo Corsini - Via della Lungara
10, 00165 Roma, ITALY,}\texttt{ vieri.benci@unipi.it}} \and Lorenzo Luperi Baglini\thanks{\textsc{L. Luperi Baglini, Faculty of Mathematics, University of Vienna,
Austria, Oskar-Morgenstern-Platz 1, 1090 Wien, AUSTRIA,}\texttt{ lorenzo.luperi.baglini@univie.ac.at}}\thanks{L.~Luperi Baglini has been supported by grants P25311-N25 and M1876-N35
of the Austrian Science Fund FWF.}}
\maketitle
\begin{abstract}
In many situations, the notion of function is not sufficient and it
needs to be extended. A classical way to do this is to introduce the
notion of weak solution; another approach is to use generalized functions.
Ultrafunctions are a particular class of generalized functions that
has been previously introduced and used to define generalized solutions
of stationary problems in \cite{ultra,belu2012,milano,beyond,gauss}.
In this paper we generalize this notion in order to study also evolution
problems. In particular, we introduce the notion of Generalized Ultrafunction
Solution (GUS) for a large family of PDE's, and we confront it with
classical strong and weak solutions. Moreover, we prove an existence
and uniqueness result of GUS for a large family of PDE's, including
the nonlinear Schroedinger equation and the nonlinear wave equation.
Finally, we study in detail GUS of Burgers' equation, proving that
(in a precise sense) the GUS of this equation provides a description
of the phenomenon at microscopic level.
\end{abstract}
\tableofcontents{}

\section{Introduction}

In order to solve many problems of mathematical physics, the notion
of function is not sufficient and it is necessary to extend it. Among
people working in partial differential equations, the theory of distributions
of Schwartz and the notion of weak solution are the main tools to
be used when equations do not have classical solutions. Usually, these
equations do not have classical solutions since they develop singularities.
The notion of weak solution allows to obtain existence results, but
uniqueness may be lost; also, these solutions might violate the conservation
laws. As an example let us consider the Burgers' equation: 
\begin{equation}
\frac{\partial u}{\partial t}+u\frac{\partial u}{\partial x}=0.\tag{BE}\label{BE}
\end{equation}
A local classical solution $u(t,x)$ is unique and, if it has compact
support, it preserves the momentum $P=\int u\ dx$ and the energy
$E=\frac{1}{2}\int u^{2}\ dx$ as well as other quantites. However,
at some time a singularity appears and the solution can be no longer
described by a smooth function. The notion of weak solution is necessary,
but the problem of uniqueness becomes a central issue. Moreover, in
general, $E$ is not preserved.

An approach that can be used to try to overcome these difficulties
is the use of generalized functions (see e.g. \cite{biagioni,biagioni-MO,ros},
where such an approach is developed using ideas in common with Colombeau
theory). In this paper we use a similar approach by means of non-Archimedean
analysis, and we introduce the notion of ultrafunction solution for
a large family of PDE's using some of the tools of Nonstandard Analysis
(NSA). Ultrafunctions are a family of generalized functions defined
on the field of hyperreals, which are a well known extension of the
reals. They have been introduced in \cite{ultra}, and also studied
in \cite{belu2012,belu2013,algebra,beyond,gauss,topologia}. The non-Archimedean
setting in which we will work (which is a reformulation, in a topological
language, of the ultrapower approach to NSA of Keisler) is introduced
in Section \ref{lt}. In Section \ref{ul} we introduce the spaces
of ultrafunctions, and we show their relationships with distributions.
In Section \ref{gus} we introduce the notion of generalized ultrafunction
solutions (GUS). We prove an existence and uniqueness theorem for
these generalized solutions, and we confront them with strong and
weak solutions of evolution problems. In particular, we show the existence
of a GUS even in the presence of blow ups (as e.g. in the case of
the nonlinear Schroedinger equation), and we show the uniqueness of
GUS for the nonlinear wave equation. Finally, in Section \ref{tbe}
we study in detail Burgers' equation and, in a sense precised in Section
\ref{sub:The-microscopic-part}, we show that in this case the unique
GUS of this equation provides a description of the phenomenon at microscopic
level.

\subsection{Notations\label{not}}

Let $\Omega$\ be a subset of $\mathbb{R}^{N}$: then 
\begin{itemize}
\item $\mathcal{C}\left(\Omega\right)$ denotes the set of continuous functions
defined on $\Omega\subset\mathbb{R}^{N};$ 
\item $\mathcal{C}_{c}\left(\Omega\right)$ denotes the set of continuous
functions in $\mathcal{C}\left(\Omega\right)$ having compact support
in $\Omega;$ 
\item $\mathcal{C}_{0}\left(\overline{\Omega}\right)$ denotes the set of
continuous functions in $\mathcal{C}\left(\Omega\right)$ which vanish
on $\partial\Omega;$ 
\item $\mathcal{C}^{k}\left(\Omega\right)$ denotes the set of functions
defined on $\Omega\subset\mathbb{R}^{N}$ which have continuous derivatives
up to the order $k;$ 
\item $\mathcal{C}_{c}^{k}\left(\Omega\right)$ denotes the set of functions
in $\mathcal{C}^{k}\left(\Omega\right)\ $having compact support; 
\item $\mathscr{D}\left(\Omega\right)$ denotes the set of the infinitely
differentiable functions with compact support defined on $\Omega\subset\mathbb{R}^{N};\ \mathcal{\mathscr{D}}^{\prime}\left(\Omega\right)$
denotes the topological dual of $\mathscr{D}\left(\Omega\right)$,
namely the set of distributions on $\Omega;$ 
\item for any set $X$, $\mathcal{P}_{fin}(X)$ denotes the set of finite
subsets of $X$; 
\item if $W$ is a generic function space, its topological dual will be
denated by $W^{\prime}$ and the paring by $\left\langle \cdot,\cdot\right\rangle _{W}$,
or simply by $\left\langle \cdot,\cdot\right\rangle .$ 
\end{itemize}

\section{$\Lambda$-theory\label{lt}}

\subsection{Non-Archimedean Fields\label{naf}}

In this section we recall the basic definitions and facts regarding
non-Archimedean fields, following an approach that has been introduced
in \cite{topologia} (see also \cite{ultra,BDN2003,belu2012,belu2013,milano,algebra,beyond,gauss}).
In the following, ${\mathbb{K}}$ will denote an ordered field. We
recall that such a field contains (a copy of) the rational numbers.
Its elements will be called numbers. 
\begin{defn}
Let $\mathbb{K}$ be an infinite ordered field. Let $\xi\in\mathbb{K}$.
We say that:
\begin{itemize}
\item $\xi$ is infinitesimal if, for all positive $n\in\mathbb{N}$, $|\xi|<\frac{1}{n}$; 
\item $\xi$ is finite if there exists $n\in\mathbb{N}$ such that $|\xi|<n$; 
\item $\xi$ is infinite if, for all $n\in\mathbb{N}$, $|\xi|>n$ (equivalently,
if $\xi$ is not finite). 
\end{itemize}
An ordered field $\mathbb{K}$ is called non-Archimedean if it contains
an infinitesimal $\xi\neq0$. 
\end{defn}
It's easily seen that all infinitesimal are finite, that the inverse
of an infinite number is a nonzero infinitesimal number, and that
the inverse of a nonzero infinitesimal number is infinite. 
\begin{defn}
A superreal field is an ordered field $\mathbb{K}$ that properly
extends $\mathbb{R}$. 
\end{defn}
It is easy to show, due to the completeness of $\mathbb{R}$, that
there are nonzero infinitesimal numbers and infinite numbers in any
superreal field. Infinitesimal numbers can be used to formalize a
notion of \textquotedbl{}closeness\textquotedbl{}: 
\begin{defn}
\label{def infinite closeness} We say that two numbers $\xi,\zeta\in{\mathbb{K}}$
are infinitely close if $\xi-\zeta$ is infinitesimal. In this case
we write $\xi\sim\zeta$. 
\end{defn}
Clearly, the relation \textquotedbl{}$\sim$\textquotedbl{} of infinite
closeness is an equivalence relation. 
\begin{thm}
If $\mathbb{K}$ is a superreal field, every finite number $\xi\in\mathbb{K}$
is infinitely close to a unique real number $r\sim\xi$, called the
\textbf{shadow} or the \textbf{standard part} of $\xi$. 
\end{thm}
Given a finite number $\xi$, we denote its shadow as $sh(\xi)$.

\subsection{The $\Lambda-$limit}

In this section we introduce a particular non-Archimedean field by
means of $\Lambda-$theory\footnote{Readers expert in nonstandard analysis will recognize that $\Lambda$-theory
is equivalent to the superstructure constructions of Keisler (see
\cite{keisler76} for a presentation of the original constructions
of Keisler, and \cite{topologia} for a comparison between these two
approaches to nonstandard analysis). }, in particular by means of the notion of $\Lambda-$limit (for complete
proofs and for further informations the reader is referred to \cite{benci99},
\cite{ultra}, \cite{belu2012} and \cite{topologia}). To recall
the basics of $\Lambda-$theory we have to recall the notion of superstructure
on a set (see also \cite{keisler76}): 
\begin{defn}
Let $E$ be an infinite set. The superstructure on $E$ is the set
\[
V_{\infty}(E)=\bigcup_{n\in\mathbb{N}}V_{n}(E),
\]
where the sets $V_{n}(E)$ are defined by induction by setting 
\[
V_{0}(E)=E
\]
and, for every $n\in\mathbb{N}$, 
\[
V_{n+1}(E)=V_{n}(E)\cup\mathcal{P}\left(V_{n}(E)\right).
\]

\end{defn}
Here $\mathcal{P}\left(E\right)$ denotes the power set of $E.$ Identifying
the couples with the Kuratowski pairs and the functions and the relations
with their graphs, it follows that $V_{\infty}(E)$ contains almost
every usual mathematical object that can be constructed starting with
$E;$ in particular, $V_{\infty}(\mathbb{R})$, which is the superstructure
that we will consider in the following, contains almost every usual
mathematical object of analysis.

Throughout this paper we let 
\[
\mathfrak{L}=\mathcal{P}_{fin}(V_{\infty}(\mathbb{R}))
\]
and we order $\mathfrak{L}$ via inclusion. Notice that $(\mathfrak{L},\subseteq)$
is a directed set. We add to $\mathfrak{L}$ a \textquotedbl{}point
at infinity\textquotedbl{} $\Lambda\notin\mathfrak{L}$, and we define
the following family of neighborhoods of $\Lambda:$ 
\[
\{\{\Lambda\}\cup Q\mid Q\in\mathcal{U}\},
\]
where $\mathcal{U}$ is a fine ultrafilter on $\mathfrak{L}$, namely
a filter such that 
\begin{itemize}
\item for every $A,B\subseteq\mathfrak{L}$, if $A\cup B=\mathfrak{L}$
then $A\in\mathcal{U}$ or $B\in\mathcal{U}$; 
\item for every $\lambda\in\mathfrak{L}$ the set $I_{\lambda}=\{\mu\in\mathfrak{L}\mid\lambda\subseteq\mu\}\in\mathcal{U}$. 
\end{itemize}
In particular, we will refer to the elements of $\mathcal{U}$ as
qualified sets and we will write $\Lambda=\Lambda(\mathcal{U})$ when
we want to highlight the choice of the ultrafilter. We are interested
in considering real nets with indices in $\mathfrak{L}$, namely functions
\[
\varphi:\mathfrak{L}\rightarrow\mathbb{R}\text{.}
\]

In particular, we are interested in $\Lambda-$limits of these nets,
namely in 
\[
\lim_{\lambda\rightarrow\Lambda}\varphi(\lambda).
\]

The following has been proved in \cite{topologia}. 
\begin{thm}
\label{nuovo}There exists a non-Archimedean superreal field $(\mathbb{K},+,\cdot,<)$
and an Hausdorff topology $\tau$ on the space $\left(\mathfrak{L}\times\mathbb{R}\right)\cup\mathbb{K}$
such that 
\begin{enumerate}
\item \label{tre}$\left(\mathfrak{L}\times\mathbb{R}\right)\cup\mathbb{K}=cl_{\tau}\left(\mathfrak{L}\times\mathbb{R}\right);$ 
\item \label{uno}for every net $\varphi:\mathfrak{L}\rightarrow\mathbb{R}$
the limit 
\[
L=\lim_{\lambda\rightarrow\Lambda}(\lambda,\varphi(\lambda))
\]
exists, it is in $\mathbb{K}$ and it is unique; moreover for every
$\xi\in\mathbb{K}$ there is a net $\varphi:\mathfrak{L}\rightarrow\mathbb{R}$
such that 
\[
\xi=\lim_{\lambda\rightarrow\Lambda}(\lambda,\varphi(\lambda));
\]

\item \label{due}$\forall$ $c\in\mathbb{R}$ we have that 
\[
\lim_{\lambda\rightarrow\Lambda}\left(\lambda,c\right)=c;
\]

\item \label{quattro}for every $\varphi,\psi:\mathfrak{L}\rightarrow\mathbb{R}$
we have that 
\begin{eqnarray*}
\lim_{\lambda\rightarrow\Lambda}\left(\lambda,\varphi(\lambda)\right)+\lim_{\lambda\rightarrow\Lambda}\left(\lambda,\psi(\lambda)\right) & = & \lim_{\lambda\rightarrow\Lambda}\left(\lambda,(\varphi+\psi)(\lambda)\right);\\
\lim_{\lambda\rightarrow\Lambda}\left(\lambda,\varphi(\lambda)\right)\cdot\lim_{\lambda\rightarrow\Lambda}\left(\lambda,\varphi(\lambda)\right) & = & \lim_{\lambda\rightarrow\Lambda}\left(\lambda,(\varphi\cdot\psi)(\lambda)\right).
\end{eqnarray*}

\end{enumerate}
\end{thm}
\begin{proof}
For a complete proof of Theorem \ref{nuovo} we refer to \cite{topologia}.
The idea\footnote{To work, this idea needs some additional requirement on the ultrafilter
$\mathcal{U}$, see e.g. \cite{BDN200301}, \cite{topologia}.} is to set 
\[
I=\left\{ \varphi\in\mathfrak{F}\left(\mathfrak{L},\mathbb{R}\right)\ |\ \varphi(x)=0\ \text{in a qualified set}\right\} ;
\]
it is not difficult to prove that $I$ is a maximal ideal in $\mathfrak{F}\left(\mathfrak{L},\mathbb{R}\right),$
and hence 
\[
\mathbb{K}:=\frac{\mathfrak{F}\left(\mathfrak{L},\mathbb{R}\right)}{I}
\]
is a field. Now the claims of Theorem \ref{nuovo} follows by identifying
every real number $c\in\mathbb{R}$ with the equivalence class of
the constant net $\left[c\right]_{I}$ and by taking the topology
$\tau$ generated by the basis of open sets 
\[
b(\tau)=\left\{ N_{\varphi,Q}\ |\ \varphi\in\mathfrak{F}\left(\mathfrak{L},\mathbb{R}\right),Q\in\mathcal{U}\right\} \cup\mathcal{P}(\mathfrak{L}\times\mathbb{R}),
\]
where 
\[
N_{\varphi,Q}:=\left\{ \left(\lambda,\varphi(\lambda)\right)\mid\lambda\in Q\right\} \cup\left\{ \left[\varphi\right]_{I}\right\} 
\]
is a neighborhood of $\left[\varphi\right]_{I}$. 
\end{proof}
Now we want to define the $\Lambda$-limit of nets $(\lambda,\varphi(\lambda))_{\lambda\in\mathfrak{L}}$,
where $\varphi(\lambda)$ is any bounded net of mathematical objects
in $V_{\infty}(\mathbb{R)}$ (a net $\varphi:\mathfrak{L}\rightarrow V_{\infty}(\mathbb{R)}$
is called bounded if there exists $n$ such that $\forall\lambda\in\mathfrak{L},\ \varphi(\lambda)\in V_{n}(\mathbb{R)}$).
To this aim, let us consider a net 
\begin{equation}
\varphi:\mathfrak{L}\rightarrow V_{n}(\mathbb{R}).\label{net}
\end{equation}
We will define $\lim_{\lambda\rightarrow\Lambda}\ \left(\lambda,\varphi(\lambda)\right)$
by induction on $n$. 
\begin{defn}
\label{def}For $n=0,$ $\lim\limits _{\lambda\rightarrow\Lambda}(\lambda,\varphi(\lambda))$
exists by Thm. (\ref{nuovo}); so by induction we may assume that
the limit is defined for $n-1$ and we define it for the net (\ref{net})
as follows: 
\[
\lim_{\lambda\rightarrow\Lambda}\ \left(\lambda,\varphi(\lambda)\right)=\left\{ \lim\limits _{\lambda\rightarrow\Lambda}(\lambda,\psi(\lambda))\ |\ \psi:\mathfrak{L}\rightarrow V_{n-1}(\mathbb{R)}\text{ and}\ \forall\lambda\in\mathfrak{L},\ \psi(\lambda)\in\varphi(\lambda)\right\} .
\]

\end{defn}
From now on, we set 
\[
\lim\limits _{\lambda\uparrow\Lambda}\varphi(\lambda):=\lim_{\lambda\rightarrow\Lambda}\left(\lambda,\varphi(\lambda)\right).
\]
Notice that it follows from Definition \ref{def} that $\lim\limits _{\lambda\uparrow\Lambda}\varphi(\lambda)$
is a well defined object in $V_{\infty}(\mathbb{R}^{\ast})$ for every
bounded net $\varphi:\mathfrak{L}\rightarrow V_{\infty}(\mathbb{R})$.

\subsection{Natural extension of sets and functions}

In this section we want to show how to extend subsets and functions
defined on $\mathbb{R}$ to subsets and functions defined on $\mathbb{K}$. 
\begin{defn}
\label{janez}Given \textit{\emph{a set}}\textit{ }$E\subseteq\mathbb{R}$,
we set 
\[
E^{\ast}:=\left\{ \lim_{\lambda\uparrow\Lambda}\psi(\lambda)\ |\ \forall\lambda\in\mathfrak{L}\,\psi(\lambda)\in E\right\} .
\]
$E^{\ast}$ is called the \textbf{natural extension }of $E.$ 
\end{defn}
Thus $E^{\ast}$ is the set of all the limits of nets with values
in $E$. Following the notation introduced in Def. \ref{janez}, from
now on we will denote $\mathbb{K}$ by $\mathbb{R}^{\ast}.$ Similarly,
it is possible to extend functions. 
\begin{defn}
\label{extfun}Given a function 
\[
f:A\rightarrow B
\]
we call natural extension of $f$ the function 
\[
f^{\ast}:A^{\ast}\rightarrow B^{\ast}
\]
such that 
\[
f^{\ast}\left(\lim_{\lambda\rightarrow\Lambda}\left(\lambda,\varphi(\lambda)\right)\right):=\lim_{\lambda\rightarrow\Lambda}\left(\lambda,f\left(\varphi(\lambda)\right)\right)
\]
for every $\varphi:\mathfrak{L}\rightarrow A.$ 
\end{defn}
That Definition \ref{extfun} is well posed has been proved in \cite{topologia}.
Let us observe that, in particular, $f^{\ast}(a)=f(a)$ for every
$a\in A$ (which is why $f^{\ast}$ is called the extension of $f$).

\section{Ultrafunctions\label{ul}}

\subsection{Definition of Ultrafunctions}

We follow the construction of ultrafunctions that we introduced in
\cite{gauss}. Let $N$ be a natural number, let $\Omega$ be a subset
of $\mathbb{R}^{N}$ and let $V\subset\mathfrak{F}\left(\Omega,\mathbb{R}\right)$
be a function vector space such that $\mathcal{\mathscr{D}}(\Omega)\subseteq V(\Omega)\subseteq L^{2}(\Omega).$ 
\begin{defn}
\label{approxseq} We say that $(V_{\lambda}(\Omega)_{\lambda\in\mathfrak{L}})$
is an \textbf{approximating net for} $V(\Omega)$ if\end{defn}
\begin{enumerate}
\item $V_{\lambda}(\Omega)$ is a finite dimensional vector subspace of
$V(\Omega)$ for every $\lambda\in\mathfrak{L}$; 
\item if $\lambda_{1}\subseteq\lambda_{2}$ then $V_{\lambda_{1}}(\Omega)\subseteq V_{\lambda_{2}}(\Omega)$; 
\item \label{unioni} if $W(\Omega)\subset V(\Omega)$ is a finite dimensional
vector space then there exists $\lambda\in\mathfrak{L}$ such that
$W(\Omega)\subseteq V_{\lambda}(\Omega)$ (i.e., $V(\Omega)=\bigcup\limits _{\lambda\in\mathfrak{L}}V_{\lambda}(\Omega)$). 
\end{enumerate}
Let us show two examples. 
\begin{example}
Let $V(\mathbb{R})\subseteq L^{2}(\mathbb{R})$. We set, for every
$\lambda\in\mathfrak{L}$,

\[
V_{\lambda}(\Omega):=Span(V(\Omega)\cap\lambda).
\]
Then $(V_{\lambda}(\Omega))_{\lambda\in\mathfrak{L}}$ is an approximating
net for $V(\Omega)$. 
\end{example}

\begin{example}
Let 
\[
\left\{ e_{a}\right\} _{a\in\mathbb{R}}
\]
be a Hamel basis\footnote{We recall that $\left\{ e_{a}\right\} _{a\in\mathbb{R}}$ is a Hamel
basis for $W$ if $\left\{ e_{a}\right\} _{a\in\mathbb{R}}$ is a
set of linearly indipendent elements of $W$ and every element of
$W$ can be (uniquely) written has a finite sum (with coefficients
in $\mathbb{R}$) of elements of $\left\{ e_{a}\right\} _{a\in\mathbb{R}}.$
Since a Hamel basis of $W$ has the continuum cardinality we can use
the points of $\mathbb{R}$ as indices for this basis.

{}} of $V(\Omega)\subseteq L^{2}$. For every $\lambda\in\mathfrak{L}$
let 
\[
V_{\lambda}(\Omega)=Span\left\{ e_{a}\ |\ a\in\lambda\right\} .
\]
Then $(V_{\lambda}(\Omega))_{\lambda\in\mathfrak{L}}$ is an approximating
net for $V(\Omega).$ \end{example}
\begin{defn}
Let $\mathcal{U}$ be a fine ultrafilter on $\mathfrak{L}$, let $\Lambda=\Lambda(\mathcal{U})$
and let $(V_{\lambda}(\Omega))_{\lambda\in\mathfrak{L}}$ be an approximating
net for $V(\Omega)$. We call \textbf{space of ultrafunctions} \textbf{generated
by }$(V_{\lambda}(\Omega))$ the $\Lambda$-limit 
\[
V_{\Lambda}(\Omega):=\lim_{\lambda\uparrow\Lambda}V_{\lambda}(\Omega)=\left\{ \lim_{\lambda\uparrow\Lambda}f_{\lambda}\ |\ \forall\lambda\in\mathfrak{L~}f_{\lambda}\in V_{\lambda}(\Omega)\right\} .
\]
In this case we will also say that the space $V_{\Lambda\text{ }}(\Omega)$
is based on the space $V(\Omega)$. When $V_{\lambda}(\Omega):=Span(V(\Omega)\cap\lambda)$
for every $\lambda\in\mathfrak{L}$, we will say that $V_{\Lambda}(\Omega)$
is a \textbf{canonical space of ultrafunctions}. 
\end{defn}
Using the above definition, if $V(\Omega)$, $\Omega\subset\mathbb{R}^{N}$,
is a real function space and $(V_{\lambda}(\Omega))$ is an approximating
net for $V(\Omega)$ then we can associate to $V(\Omega)$ the following
three hyperreal functions spaces: 
\begin{equation}
V(\Omega)^{\sigma}=\left\{ f^{\ast}\ |\ f\in V(\Omega)\right\} ;\label{sigma1}
\end{equation}
\begin{equation}
V_{\Lambda}(\Omega)=\left\{ \lim_{\lambda\uparrow\Lambda}\ f_{\lambda}\ |\ \forall\lambda\in\mathfrak{L~}f_{\lambda}\in V_{\lambda}(\Omega)\right\} ;\label{tilda}
\end{equation}

\begin{equation}
V(\Omega)^{\ast}=\left\{ \lim_{\lambda\uparrow\Lambda}\ f_{\lambda}\ |\ \forall\lambda\in\mathfrak{L~}f_{\lambda}\in V(\Omega)\right\} .\label{star}
\end{equation}
Clearly we have 
\[
V(\Omega)^{\sigma}\subset V_{\Lambda}(\Omega)\subset V(\Omega)^{\ast}.
\]

So, given any vector space of functions $V(\Omega)$, the space of
ultrafunctions generated by $V(\Omega)$ is a vector space of hyperfinite
dimension that includes $V(\Omega)^{\sigma}$, and the ultrafunctions
are $\Lambda$-limits of functions in $V_{\lambda}(\Omega)$. Hence
the ultrafunctions are particular internal functions 
\[
u:\left(\mathbb{R}^{\ast}\right)^{N}\rightarrow{\mathbb{C}^{\ast}.}
\]
Since $V_{\Lambda}(\Omega)\subset\left[L^{2}(\mathbb{R})\right]^{\ast},$
we can equip $V_{\Lambda}(\Omega)$ with the following scalar product:
\begin{equation}
\left(u,v\right)=\int^{\ast}u(x)v(x)\ dx,\label{inner}
\end{equation}
where $\int^{\ast}$ is the natural extension of the Lebesgue integral
considered as a functional 
\[
\int:L^{1}(\Omega)\rightarrow{\mathbb{R}}.
\]
Therefore, the norm of an ultrafunction will be given by 
\[
\left\Vert u\right\Vert =\left(\int^{\ast}|u(x)|^{2}\ dx\right)^{\frac{1}{2}}.
\]
Sometimes, when no ambiguity is possible, in order to make the notation
simpler we will write $\int$ istead of $\int^{\ast}$. 
\begin{rem}
\label{nina}Notice that the natural extension $f^{\ast}$ of a function
$f$ is an ultrafunction if and only if $f\in V(\Omega).$ \end{rem}
\begin{proof}
Let $f\in V(\Omega)$ and let $(V_{\lambda}(\Omega))$ be an approximating
net for $V(\Omega)$. Then, eventually, $f\in V_{\lambda}(\Omega)$
and hence 
\[
f^{\ast}=\lim_{\lambda\uparrow\Lambda}f\in\lim_{\lambda\uparrow\Lambda}\ V_{\lambda}(\Omega)=V_{\Lambda}(\Omega).
\]
Conversely, if $f\notin V(\Omega)$ then $f^{\ast}\notin V^{\ast}(\Omega)$
and, since $V_{\Lambda}(\Omega)\subset V^{\ast}(\Omega)$, this entails
the thesis. 
\end{proof}

\subsection{Canonical extension of functions, functionals and operators}

Let $V_{\Lambda}(\Omega)$ be a space of ultrafunctions based on $V(\Omega)\subseteq L^{2}(\Omega)$.
We have seen that given a function $f\in V(\Omega),$ its natural
extension 
\[
f^{\ast}:\Omega^{\ast}\rightarrow\mathbb{R}^{\ast}
\]
is an ultrafunction in $V_{\Lambda}(\Omega).$ In this section we
investigate the possibility to associate an ultrafunction $\widetilde{f}$
to any function $f\in L_{loc}^{1}(\Omega)$ in a consistent way. Since
$L^{2}(\Omega)\subseteq V^{\prime}(\Omega)$, this association can
be done by means of a duality method. 
\begin{defn}
\label{CP}Given $T\in\left[L^{2}(\Omega)\right]^{\ast},$ we denote
by $\widetilde{T}$ the unique ultrafunction such that $\forall v\in V_{\Lambda}(\Omega),$
\[
\int_{\Omega^{\ast}}\widetilde{T}(x)v(x)\,dx=\int_{\Omega^{\ast}}T(x)v(x)\,dx.
\]
The map 
\[
P_{\Lambda}:\left[L^{2}(\Omega)\right]^{\ast}\rightarrow V_{\Lambda}(\Omega)
\]
defined by $P_{\Lambda}T=\widetilde{T}$ will be called the \textbf{canonical
projection.} 
\end{defn}
The above definition makes sense, as $T$ is a linear functional on
$V(\Omega)^{\ast}$, and hence on $V_{\Lambda}(\Omega)\subset V(\Omega)^{\ast}.$

Since $V(\Omega)\subset L^{2}(\Omega),$ using the inner product (\ref{inner})
we can identify $L^{2}(\Omega)$ with a subset of $V^{\prime}(\Omega),$
and hence $\left[L^{2}(\Omega)\right]^{\ast}$ with a subset of $\left[V^{\prime}(\Omega)\right]^{\ast};$
in this case, $\forall f\in\left[L^{2}(\Omega)\right]^{\ast},\ \forall v\in V_{\Lambda}(\Omega),$
\[
\int\widetilde{f}(x)v(x)\ dx=\int f(x)v(x)\ dx,
\]
namely the map $P_{\Lambda}f=\widetilde{f}$ restricted to $\left[L^{2}(\Omega)\right]^{\ast}$
reduces to the orthogonal projection 
\[
P_{\Lambda}:\left[L^{2}(\Omega)\right]^{\ast}\rightarrow V_{\Lambda}(\Omega).
\]
If we take any function $f\in L_{loc}^{1}(\Omega)\cap L^{2}(\Omega),$
then $f^{\ast}\in\left[L_{loc}^{1}(\Omega)\cap L^{2}(\Omega)\right]^{\ast}\subset\left[L^{2}(\Omega)\right]^{\ast}$
and hence $\widetilde{f^{\ast}}$ is well defined by Def.~\ref{CP}.
In order to simplify the notation we will simply write $\widetilde{f.}$
This discussion suggests the following definition: 
\begin{defn}
Given a function $f\in L_{loc}^{1}(\Omega)\cap L^{2}(\Omega),$ we
denote by $\widetilde{f}$ the unique ultrafunction in $V_{\Lambda}(\Omega)$
such that $\forall v\in V_{\Lambda}(\Omega),$ 
\[
\int\widetilde{f}(x)v(x)\ dx=\int f^{\ast}(x)v(x)\ dx.
\]
$\widetilde{f}$ is called the canonical extension of $f.$ \end{defn}
\begin{rem}
\label{cina}As we observed, for every $f:\mathbb{R}\rightarrow\mathbb{R}$
we have that $^{\ast}f\in V_{\Lambda}(\Omega)$ iff $f\in V(\Omega).$
Therefore for every $f:\mathbb{R}\rightarrow\mathbb{R}$ 
\[
\widetilde{f}=f^{\ast}\Leftrightarrow f\in V(\Omega).
\]

\end{rem}
Let us observe that we need to assume that $V(\Omega)\subset L_{c}^{\infty}(\Omega)=(L_{loc}^{1}(\Omega))^{\prime}$
if we want $\widetilde{f}$ to be defined for every function $f\in L_{loc}^{1}(\Omega)$.
Using a similar method, it is also possible to extend operators: 
\begin{defn}
\label{b}Given an operator 
\[
\mathcal{A}:V(\Omega)\rightarrow V^{\prime}(\Omega)
\]
we can extend it to an operator 
\[
\widetilde{\mathcal{A}}:V_{\Lambda}(\Omega)\rightarrow V_{\Lambda}(\Omega)
\]
in the following way: given an ultrafunction $u,$ $\widetilde{\mathcal{A}}(u)$
is the unique ultrafunction such that 
\[
\forall v\in V_{\Lambda}(\Omega),\ \int^{\ast}\widetilde{\mathcal{A}}(u)v\ dx\ =\int^{\ast}\mathcal{A}^{\ast}(u)v\,dx;
\]
namely 
\[
\widetilde{\mathcal{A}}=P_{\Lambda}\circ\mathcal{A}^{\ast},
\]
where $P_{\Lambda}$ is the canonical projection. 
\end{defn}
Sometimes, when no ambiguity is possible, in order to make the notation
simpler we will write $\mathcal{A}(u)$ instead of $\widetilde{\mathcal{A}}(u).$ 
\begin{example}
The derivative of an ultrafunction is well defined provided that the
weak derivative is defined from $V(\Omega)$ to his dual $V^{\prime}(\Omega):$
\[
\mathcal{\partial}:V(\Omega)\rightarrow V^{\prime}(\Omega).
\]
For example you can take $V(\Omega)=\mathcal{C}^{1}(\Omega),$ $H^{1/2}(\Omega),$
$BV(\Omega)$ etc. Following Definition \ref{b}, we have that the
ultrafunction derivative 
\[
D:V_{\Lambda}(\Omega)\rightarrow V_{\Lambda}(\Omega)
\]
of an ultrafunction $u$ is defined by duality as the unique ultrafunction
$Du$ such that 
\begin{equation}
\forall v\in V_{\Lambda}(\Omega),\ \int Du\ v\ dx\ =\left\langle \partial^{\ast}u,v\right\rangle .\label{sd}
\end{equation}
Notice that, in order to simplify the notation, we have denoted the
generalized derivative by $D=\widetilde{\mathcal{\partial}}.$ 
\end{example}
To construct the space of ultrafunctions that we need to study Burgers'
Equation we will use the following theorem: 
\begin{thm}
\label{thm:Spaces of Ultrafunctions as hyperfinite extensions}Let
$n\in\mathbb{N},$ $\Omega\subseteq\mathbb{R}^{n}$ and let $V(\Omega)$
be a vector space of functions. Let $V(\Omega)^{\ast}$ be a $\left|\mathfrak{L}\right|^{+}$-enlarged\footnote{For the notion of enlarging, as well as for other important notions
in nonstandard analysis such as saturation and overspill, we refer
to \cite{keisler76,rob}.} ultrapower of $V(\Omega)$. Then every hyperfinite dimensional vector
space $W(\Omega)$ such that $V(\Omega)^{\sigma}\subseteq W(\Omega)\subseteq V(\Omega)^{\ast}$
contains an isomorphic copy of a canonical space of ultrafunctions
on $V(\Omega)$.\end{thm}
\begin{proof}
First of all, we claim that there exist a hyperfinite set $H\in\left(\mathcal{P}_{fin}(\mathfrak{L})\right)^{\ast}$
such that $\lambda\subseteq H$ for every $\lambda\in\mathfrak{L}$
and such that $B=H\cap W(\Omega)$ is a hyperfinite basis of $W(\Omega)$.
To prove this claim we set, for every $\lambda\in\mathfrak{L}$, 
\[
H_{\lambda}=\left\{ H\in\left(\mathcal{P}_{fin}(\mathfrak{L})\right)^{\ast}\mid\lambda^{\ast}\subseteq H\,\text{and}\,Span(H\cap V^{\ast}(\Omega))=W(\Omega)\right\} .
\]
Clearly, if $H_{\lambda}\neq\emptyset$ for every $\lambda\in\mathfrak{L}$
then the family $\{H_{\lambda}\}_{\lambda\in\mathfrak{L}}$ has the
finite intersection property (as $H_{\lambda_{1}}\cap\dots\cap H_{\lambda_{k}}=H_{\lambda_{1}\cup\dots\cup\lambda_{k}})$.
To prove that $H_{\lambda}\neq\emptyset$ for every $\lambda\in\mathfrak{L}$,
let $\lambda\in\mathfrak{L}$ be given and let $B$ be a fixed hyperfinite
basis of $W(\Omega)$ with $V(\Omega)^{\sigma}\subseteq B$ (whose
existence can be easily deduced from the enlarging property of the
extension, as $V(\Omega)^{\sigma}\subseteq W(\Omega)$). Let $\lambda=\lambda_{0}\cup\lambda_{1}$,
where $\lambda_{0}\cap\lambda_{1}=\emptyset$ and $\lambda_{0}=\lambda\cap V(\Omega)$,
and let $H=B\cup\lambda_{1}^{\ast}$. It is immediate to notice that
$H\in H_{\lambda}$. Therefore this proves that the family $\{H_{\lambda}\}_{\lambda\in\mathfrak{L}}$
has the finite intersection property, and so our claim can be derived
as a consequence of the $|\mathfrak{L}|^{+}$-enlarging property of
the extension. From now on, we let $H$ be an hyperfinite set with
the properties of our claim, and we let $B=H\cap W(\Omega)$. Finally,
we set $\mathcal{U}=\{X\subseteq\mathfrak{L}\mid H\in X^{\ast}\}$.
Clearly, $\mathcal{U}$ is an ultrafilter on $\mathfrak{L}$; moreover,
our construction of $H$ has been done to have that $\mathcal{U}$
is a fine ultrafilter. To prove this, let $\lambda_{0}\in\mathfrak{L}$.
Then 
\[
\{\lambda\in\mathfrak{L}\mid\lambda_{0}\subseteq\lambda\}\in\mathcal{U}\Leftrightarrow H\in\{\lambda\in\mathfrak{L}\mid\lambda_{0}\subseteq\lambda\}^{\ast}\Leftrightarrow\lambda_{0}\subseteq H,
\]
and $\lambda_{0}\subseteq H$ by our construction of the set $H$.

Now we set $V_{\lambda}(\Omega)=Span(V(\Omega)\cap\lambda$) for every
$\lambda\in\mathfrak{L}$, we set $V_{\Lambda(\mathcal{U})}:=\lim_{\lambda\uparrow\Lambda(\mathcal{U})}v_{\lambda}$
and we let $\Phi:V_{\Lambda(\mathcal{U})}(\Omega)\rightarrow W(\Omega)$
be defined as follows: for every $v=\lim_{\lambda\uparrow\Lambda(\mathcal{U})}v_{\lambda}$,
\[
\Phi\left(\lim_{\lambda\uparrow\Lambda(\mathcal{U})}v_{\lambda}\right):=v_{B},
\]
where $v_{B}$ is the value of the hyperextension $v^{\ast}:\mathfrak{L}^{\ast}\rightarrow V^{\ast}(\Omega)$
of the function $v:\mathfrak{L}\rightarrow V(\Omega)$ evaluated in
$B\in\mathfrak{L}^{\ast}$. Let us notice that, as $v_{\lambda}\in Span(V(\Omega)\cap\lambda)$
for every $\lambda\in\mathfrak{L}$, by transfer we have that $v_{B}\in Span(V(\Omega)^{\ast}\cap B)=W(\Omega)$,
namely the image of $\Phi$ is included in $W(\Omega).$

To conclude our proof, we have to show that $\Phi$ is an embedding
(so that we can take $\Phi\left(V_{\Lambda(\mathcal{U})}(\Omega)\right)$
as the isomorphic copy of a canonical space of ultrafunctions contained
in $W(\Omega)$). The linearity of $\Phi$ holds trivially; to prove
that $\Phi$ is injective let $v=\lim_{\lambda\uparrow\Lambda(\mathcal{V})}v_{\lambda},\,w=\lim_{\lambda\uparrow\Lambda(\mathcal{V})}w_{\lambda}$.
Then 
\[
\Phi(v)=\Phi(w)\Leftrightarrow v_{B}=w_{B}\Leftrightarrow B\in\left\{ \lambda\in\mathfrak{L}\mid v_{\lambda}=w_{\lambda}\right\} ^{\ast}\Leftrightarrow
\]
\[
\left\{ \lambda\in\mathfrak{L}\mid v_{\lambda}=w_{\lambda}\right\} \in\mathcal{V}\Leftrightarrow v=w.\qedhere
\]
\end{proof}
\begin{lem}
\label{lem:Adding a function}Let $V(\Omega)$ be given, let $(V_{\lambda}(\Omega))_{\lambda\in\mathfrak{L}}$
be an approximating net for $V(\Omega)$ and let $V_{\Lambda}(\Omega)=\lim_{\lambda\uparrow\Lambda}V_{\lambda}(\Omega)$.
Finally, let $u\in V(\Omega)^{\ast}\setminus V_{\Lambda}(\Omega)$.
Then $W(\Omega):=Span\left(V_{\Lambda}(\Omega)\cup\left\{ u\right\} \right)$
is a space of ultrafunctions on $V(\Omega)$.\end{lem}
\begin{proof}
Let $u=\lim_{\lambda\uparrow\Lambda}u_{\lambda}$, where $u_{\lambda}\notin V_{\lambda}(\Omega)$
for every $\lambda\in\mathfrak{L}$, and let, for every $\lambda\in\mathfrak{L}$,
$W_{\lambda}=Span\left(V_{\lambda}\cup\left\{ u_{\lambda}\right\} \right)$.
Clearly, $(W_{\lambda})_{\lambda\in\mathfrak{L}}$ is an approximating
net for $V(\Omega)$. We claim that $W(\Omega)=W_{\Lambda}(\Omega)=\lim_{\lambda\uparrow\Lambda}W_{\lambda}(\Omega)$.
Clearly, $V_{\Lambda}(\Omega)\subseteq W_{\Lambda}(\Omega)$ and $u\in W_{\Lambda}(\Omega)$,
and hence $W(\Omega)\subseteq W_{\Lambda}(\Omega)$. As for the reverse
inclusion, let $w\in W_{\Lambda}(\Omega)$ and let $w=\lim_{\lambda\uparrow\Lambda}w_{\lambda}$.
For every $\lambda\in\mathfrak{L}$ let $w_{\lambda}=v_{\lambda}+c_{\lambda}u_{\lambda}$,
where $v_{\lambda}\in V_{\lambda}$. Then 
\[
w=\lim_{\lambda\uparrow\Lambda}v_{\lambda}+\lim_{\lambda\uparrow\Lambda}c_{\lambda}\cdot\lim_{\lambda\uparrow\Lambda}u_{\lambda}
\]
so, as $\lim_{\lambda\uparrow\Lambda}v_{\lambda}\in V_{\Lambda}(\Omega)$
and $\lim_{\lambda\uparrow\Lambda}u_{\lambda}=u$, we have that $w\in W(\Omega)$,
and hence the thesis is proved. \end{proof}
\begin{thm}
\label{vecchiaroba}There is a space of ultrafunctions $U_{\Lambda}(\mathbb{R})$
which satisfies the following assumptions: 
\begin{enumerate}
\item \label{enu: Thm 23 (1)}$H_{c}^{1}(\mathbb{R})\subseteq U_{\Lambda}(\mathbb{R})$; 
\item \label{enu: Thm 23 (2)}the ultrafunction $\widetilde{1}$ is the
identity in $U_{\Lambda}(\mathbb{R}),$ namely $\forall u\in U_{\Lambda}(\mathbb{R}),$
$u\cdot\widetilde{1}=u;$ 
\item \label{enu: Thm 23 (3)}$D\widetilde{1}=0$; 
\item \label{enu: Thm 23 (4)}$\forall u,v\in U_{\Lambda}(\mathbb{R}),$
$\int^{\ast}\left(Du\right)v\ dx\ =-\int^{\ast}u\left(Dv\right)\ dx.$ 
\end{enumerate}
\end{thm}
\begin{proof}
We set 
\begin{multline*}
H_{\flat}^{1}(\mathbb{R})=Span\{u\in L^{2}(\mathbb{R})\ |\ \exists n\in\mathbb{N}~\text{s.t}.~supp(u)\subseteq\left[-n,n\right],\\
u(n)=u(-n),~u\in H^{1}([-n,n])\}.
\end{multline*}
Let $\beta\in\mathbb{N}^{\ast}\setminus\mathbb{N}$; we set 
\[
W(\mathbb{R}):=\left\{ v\in\left[H_{\flat}^{1}(\mathbb{R})\right]^{\ast}\mid supp(u)\subseteq[-\beta,\beta],\,u(-\beta)=u(\beta)\right\} 
\]
and we let $V_{\Lambda}(\mathbb{R})$ be a hyperfinite dimensional
vector space that contains the characteristic function $1_{[-\beta,\beta]}(x)$
of $[-\beta,\beta${]} and such that\footnote{To have this property we need the nonstandard extension to be a $\left|\mathcal{P}(\mathbb{R})\right|^{+}$-enlargment.}
\[
\left[H_{\flat}^{1}(\mathbb{R})\right]^{\sigma}\subseteq V_{\Lambda}(\mathbb{R})\subseteq W(\mathbb{R}).
\]
As $W(\mathbb{R})\subseteq\left[H_{\flat}^{1}(\mathbb{R})\right]^{\ast}$
we can apply Thm. \ref{thm:Spaces of Ultrafunctions as hyperfinite extensions}
to deduce that $V_{\Lambda}(\mathbb{R})$ contains an isomorphic copy
of a canonical space of ultrafunctions on $H_{\flat}^{1}(\mathbb{R})$.
If this isomorphic copy does not contain $1_{[-\beta,\beta]}$, we
can apply Lemma \ref{lem:Adding a function} to construct a space
of ultrafunctions included in $V_{\Lambda}(\Omega)$ that contains
$1_{[-\beta,\beta]}$. Let $U_{\Lambda}(\Omega)$ denote this space
of ultrafunctions on $H_{\flat}^{1}(\mathbb{R})$.

Condition (\ref{enu: Thm 23 (1)}) holds as $H_{c}^{1}(\mathbb{R})\subseteq H_{\flat}^{1}(\mathbb{R})$.
To prove condition (\ref{enu: Thm 23 (2)}) let us show that $\widetilde{1}=1_{[-\beta,\beta]}:$
in fact, for every $u\in U_{\Lambda}(\mathbb{R})$ we have 
\[
\int\widetilde{1}\cdot u\,dx=\int1\cdot u\,dx=\int_{-\beta}^{\beta}u\,dx=\int1_{[-\beta,\beta]}\cdot u\,dx.
\]
Henceforth condition (\ref{enu: Thm 23 (2)}) holds as $1_{[-\beta,\beta]}\cdot u=u$
for every $u\in U_{\Lambda}(\mathbb{R})$. To prove condition (\ref{enu: Thm 23 (3)})
let $u\in U_{\Lambda}(\mathbb{R}).$ Then 
\[
\int D\left(1_{[-\beta,\beta]}\right)\cdot u\,dx=\int\partial\left(1_{[-\beta,\beta]}\right)\cdot u\,dx=u(\beta)-u(-\beta)=0,
\]
namely $D\left(1_{[-\beta,\beta]}\right)=0$. Finally, as $U_{\Lambda}(\mathbb{R})\subseteq\left[BV(\mathbb{R})\right]^{\ast}$,
by equation (\ref{sd}), we have that 
\[
\int Du\ v\ dx\ =\left\langle \partial^{\ast}u,v\right\rangle =-\left\langle u,\partial^{\ast}v\right\rangle =-\int u\ Dv\ dx
\]
and so condition (\ref{enu: Thm 23 (4)}) holds. \end{proof}
\begin{rem}
Let $U_{\Lambda}(\mathbb{R})$ be the space of ultrafunctions given
by Theorem \ref{vecchiaroba}. Then for every ultrafunction $u\in U_{\Lambda}(\mathbb{R})$
we have 
\[
\int^{\ast}u(x)dx=\int^{\ast}u(x)\cdot1dx=\int^{\ast}u(x)\cdot\widetilde{1}dx=\int_{-\beta}^{\beta}u(x)dx.
\]
We will use this property in Section \ref{tbe} when talking about
Burgers' equation. 
\end{rem}

\subsection{Spaces of ultrafunctions involving time\label{sub:Spaces-of-ultrafunctions with time}}

Generic problems of evolution are usually formulated by equations
of the following kind: 
\begin{equation}
\partial_{t}u=\mathcal{A}(u),\label{pura}
\end{equation}
where 
\[
\mathcal{A}:V(\Omega)\rightarrow L^{2}(\Omega)
\]
is a differential operator.

By definition, a \textbf{strong solution} of equation (\ref{pura})
is a function

\[
\phi\in V(I\times\Omega):=\mathcal{C}{}^{0}(I,V(\Omega))\cap\mathcal{C}^{1}(I,L^{2}(\Omega))
\]
where $I:=\left[0,T\right)$ is the interval of time and $\mathcal{C}^{k}(I,B),$
$k\in\mathbb{N}$, denotes the space of functions from $I$ to a Banach
space $B$ which are $k$ times differentiable with continuity.

In equation (\ref{pura}), the independent variable is $(t,x)\in I\times\Omega\subset\mathbb{R}^{N+1},$
$I=\left[0,T\right)$. A disappointing fact is that a ultrafunction
space based on $V(I\times\Omega)$ is not a convenient space where
to study this equation, since these ultrafunctions spaces are not
homogeneous in time in the following sense: if for every $t\in I^{\ast}$
we set 
\[
V_{\Lambda,t}(\Omega)=\left\{ v\in V(\Omega)^{\ast}\ |\ \exists u\in V_{\Lambda}(I\times\Omega):u(t,x)=v(x)\right\} ,
\]
for $t_{2}\neq t_{1}$ we have that 
\[
V_{\Lambda,t_{2}}(\Omega)\neq V_{\Lambda,t_{1}}(\Omega).
\]
This fact is disappointing since we would like to see $u(t,\cdot)$
as a function defined on the same space for all the times $t\in I^{\ast}$.
For this reason we think that a convenient space to study equation
(\ref{pura}) in the framework of ultrafunctions is 
\[
\mathcal{C}^{1}(I^{\ast},V_{\Lambda}(\Omega)),
\]
defined as follows: 
\begin{defn}
\label{def:C kappa}For every $k\in\mathbb{N}$ we set 
\[
\mathcal{C}^{k}(I^{\ast},V_{\Lambda}(\Omega))=\left\{ u\in\left[\mathcal{C}^{k}(I,V(\Omega))\right]^{\ast}\ |\ \forall t\in I^{\ast},\ \forall i\leq k,\,\partial_{t}^{i}u(t,\cdot)\in V_{\Lambda}(\Omega)\right\} ,\ k\in\mathbb{N}.
\]

\end{defn}
The advantage in using $\mathcal{C}^{1}(I^{\ast},V_{\Lambda}(\Omega))$
rather than $V_{\Lambda}(I\times\Omega)$ relays in the fact that
we want to consider our evolution problem as a dynamical system on
$V_{\Lambda}(\Omega)$, and the time as a continuous and homogeneous
variable. In fact, at least in the models which we will consider,
we have a better description of the phenomena in $\mathcal{C}^{1}(I^{\ast},V_{\Lambda}(\Omega))$
rather than in $V_{\Lambda}(I\times\Omega)$ or in the standard space
$\mathcal{C}^{0}(I,V(\Omega))\cap\mathcal{C}^{1}(I,L^{2}(\Omega))$.

\subsection{Ultrafunctions and distributions\label{distri}}

One of the most important properties of spaces of ultrafunctions is
that they can be seen (in some sense that we will make precise later)
as generalizations of the space of distributions (see also \cite{algebra},
where we construct an algebra of ultrafunctions that extends the space
of distributions). The proof of this result is the topic of this section.

Let $E\subset\mathbb{R}^{N}$ be a set not necessarily open. In the
applications in this paper $E$ will be $\Omega\subset\mathbb{R}^{N}$
or $\left[0,T\right)\times\Omega\subset\mathbb{R}^{N+1}.$ 
\begin{defn}
\label{DEfCorrespondenceDistrUltra}The space of \textbf{generalized
distribution} on $E$ is defined as follows: 
\[
\mathcal{\mathscr{D}}_{G}^{\prime}(E)=L^{2}(E)^{\ast}/N,
\]
where 
\[
N=\left\{ \tau\in L^{2}(E)^{\ast}\ |\ \forall\varphi\in\mathscr{D}(E)\ \int\tau\varphi\ dx\sim0\right\} .
\]

\end{defn}
The equivalence class of $u$ in $L^{2}(E)^{\ast},$ with some abuse
of notation, will be denoted by 
\[
\left[u\right]_{\mathscr{D}}.
\]

\begin{defn}
For every (internal or external) vector space $W(E)\subset L^{2}(E)^{\ast},$
we set 
\[
\left[W(E)\right]_{B}=\left\{ u\in W(E)\ |\ \forall\varphi\in\mathfrak{\mathcal{\mathscr{D}}}(E)\,\int u\varphi\ dx\ \ \text{is\ finite}\right\} .
\]

\end{defn}

\begin{defn}
Let $\left[u\right]_{\mathfrak{\mathscr{D}}}$ be a generalized distribution.
We say that $\left[u\right]_{\mathscr{D}}$ is a bounded generalized
distribution if $u\in\left[L^{2}(E)^{\ast}\right]_{B}$. 
\end{defn}
Finally, we set

\[
\mathscr{D}_{GB}^{\prime}(E):=[\mathcal{\mathscr{D}}_{G}^{\prime}(E)]_{B}.
\]

We now want to prove that the space $\mathfrak{\mathcal{\mathscr{D}}}_{GB}^{\prime}(E)$
is isomorphic (as a vector space) to $\mathfrak{\mathcal{\mathscr{D}}}^{\prime}(E).$
To do this we will need the following lemma. 
\begin{lem}
\label{cannolo} Let $(a_{n})_{n\in\mathbb{N}}$ be a sequence of
real numbers and let $l\in\mathbb{R}$. If $\lim_{n\rightarrow+\infty}a_{n}=l$
then $sh(\lim_{\lambda\uparrow\Lambda}a_{|\lambda|})=l.$ \end{lem}
\begin{proof}
Since\textbf{\ }$\lim_{n\rightarrow+\infty}a_{n}=l,$ for every $\varepsilon\in\mathbb{R}_{>0}$
the set 
\[
I_{\varepsilon}=\{\lambda\in\mathfrak{L}\mid|l-a_{|\lambda|}|<\varepsilon\}\in\mathcal{U}.
\]
In fact, let $N\in\mathbb{N}$ be such that $|a_{m}-l|<\varepsilon$
for every $m\geq N$. Then for every $\lambda_{0}\in\mathfrak{L}$
such that $|\lambda_{0}|\geq N$ we have that $I_{\varepsilon}\supseteq\{\lambda\in\mathfrak{L}\mid\lambda_{0}\subseteq\lambda\}\in\mathcal{U}$,
and this proves that $I_{\varepsilon}\in\mathcal{U}$. Therefore for
every $\varepsilon\in\mathbb{R}_{>0}$ we have 
\[
|l-\lim_{\lambda\uparrow\Lambda}a_{|\lambda|}|<\varepsilon,
\]
and so $sh(\lim_{\lambda\uparrow\Lambda}a_{|\lambda|})=l.$ \end{proof}
\begin{thm}
\label{bello}There is a linear isomorphism 
\[
\Phi:\mathfrak{\mathcal{\mathscr{D}}}_{GB}^{\prime}(E)\rightarrow\mathfrak{\mathscr{D}}^{\prime}(E)
\]
defined by the following formula: 
\[
\forall\varphi\in\mathscr{D},\,\left\langle \Phi\left(\left[u\right]_{\mathfrak{\mathcal{\mathscr{D}}}}\right),\varphi\right\rangle _{\mathfrak{\mathcal{\mathscr{D}}}(E)}=sh\left(\int^{\ast}u\ \varphi^{\ast}\ dx\right).
\]
\end{thm}
\begin{proof}
Clearly the map $\Phi$ is well defined (namely $u\approx_{\mathcal{\mathscr{D}}}v\Rightarrow\Phi\left(\left[u\right]_{\mathcal{\mathscr{D}}}\right)=\Phi\left(\left[v\right]_{\mathfrak{\mathscr{D}}}\right)$),
it is linear and its range is in $\mathcal{\mathscr{D}}^{\prime}(E)$.
It is also immediate to see that it is injective. The most delicate
part is to show that it is surjective. To see this let $T\in\mathfrak{\mathcal{\mathscr{D}}}^{\prime}(E);$
we have to find an ultrafunction $u_{T}$ such that 
\begin{equation}
\Phi\left(\left[u_{T}\right]_{\mathscr{D}}\right)=T.\label{lana}
\end{equation}
Since $L^{2}(E)$ is dense in $\mathcal{\mathscr{D}}^{\prime}(E)$
with respect to the weak topology, there is a sequence $\psi_{n}\in L^{2}(E)$
such that $\psi_{n}\rightarrow T.$ We claim that 
\[
u_{T}=\lim\limits _{\lambda\uparrow\Lambda}\ \psi_{|\lambda|}
\]
satisfies (\ref{lana}) and $[u_{T}]_{\mathcal{\mathscr{D}}}\in\mathfrak{\mathscr{D}}_{GB}^{\prime}(E)$.
Since $u_{T}$ is a $\Lambda$-limit of $L^{2}(E)$ functions, we
have that $u_{T}\in L^{2}(E)^{\ast}$, so $[u_{T}]_{\mathcal{\mathscr{D}}}\in\mathfrak{\mathcal{\mathscr{D}}}_{G}^{\prime}(E)$.
It remains to show that $[u_{T}]_{\mathfrak{\mathcal{\mathscr{D}}}}$
is bounded and that $\Phi\left(\left[u_{T}\right]_{\mathfrak{\mathscr{D}}}\right)=T$.
Take $\varphi\in\mathcal{\mathscr{D}}$; by definition, 
\[
\left\langle T,\varphi\right\rangle _{\mathfrak{\mathcal{\mathscr{D}}}(E)}=\lim\limits _{n\rightarrow+\infty}\int^{\ast}\psi_{n}\cdot\varphi dx=\lim\limits _{n\rightarrow+\infty}a_{n},
\]
where we have set $a_{n}=\int\psi_{n}\cdot\varphi dx$. Then by Lemma
\ref{cannolo} we have 
\begin{multline*}
\lim\limits _{n\rightarrow+\infty}a_{n}=sh\left(\lim\limits _{\lambda\uparrow\Lambda}a_{|\lambda|}\right)=sh\left(\lim\limits _{\lambda\uparrow\Lambda}\int\psi_{|\lambda|}\cdot\varphi dx\right)=\\
sh\left(\int^{\ast}\left(\lim\limits _{\lambda\uparrow\Lambda}\psi_{|\lambda|}\cdot\varphi\right)dx\right)=sh\left(\int^{\ast}u_{T}\cdot\varphi dx\right)=\left\langle \Phi\left(\left[u_{T}\right]_{\mathfrak{\mathscr{D}}}\right),\varphi\right\rangle _{\mathscr{D}(E)},
\end{multline*}
therefore $\left\langle \Phi\left(\left[u_{T}\right]_{\mathfrak{\mathscr{D}}}\right),\varphi\right\rangle _{\mathfrak{\mathscr{D}}(E)}=\left\langle T,\varphi\right\rangle _{\mathfrak{\mathscr{D}}(E)}\in\mathbb{R}$
and the thesis is proved. 
\end{proof}
From now on we will identify the spaces $\mathfrak{\mathscr{D}}_{GB}^{\prime}(E)\ $and
$\mathscr{D}^{\prime}(E);$ so, we will identify $\left[u\right]_{\mathscr{D}}\ $with
$\Phi\left(\left[u\right]_{\mathscr{D}}\right)$ and we will write
$\left[u\right]_{\mathscr{D}}\in\mathfrak{\mathscr{D}}^{\prime}(E)$
and 
\[
\left\langle \left[u\right]_{\mathscr{D}},\varphi\right\rangle _{\mathscr{D}(E)}:=\langle\Phi[u]_{\mathscr{D}},\varphi\rangle=sh\left(\int^{\ast}u\ \varphi^{\ast}\ dx\right).
\]

Moreover, with some abuse of notation, we will write also that $\left[u\right]_{\mathscr{D}}\in L^{2}(E),\ \left[u\right]_{\mathfrak{\mathscr{D}}}\in V(E),$
etc. meaning that the distribution $\left[u\right]_{\mathscr{D}}$
can be identified with a function $f$ in $L^{2}(E),$ $V(E),$ etc.
By our construction, this is equivalent to say that $f^{\ast}\in\left[u\right]_{\mathfrak{\mathscr{D}}}.$
So, in this case, we have that $\forall\varphi\in\mathfrak{\mathscr{D}}(E)$
\[
\left\langle \left[u\right]_{\mathfrak{\mathscr{D}}},\varphi\right\rangle _{\mathfrak{\mathscr{D}}(E)}=sh\left(\int^{\ast}u\ \varphi^{\ast}\ dx\right)=sh\left(\int^{\ast}f^{\ast}\varphi^{\ast}dx\right)=\int f\ \varphi\ dx.
\]

An immediate consequence of Theorem \ref{bello} is the following: 
\begin{prop}
The space $\left[\mathcal{C}^{1}(I,V_{\Lambda}(\Omega))\right]_{B}$
can be mapped into a space of distributions by setting, $\forall u\in\left[\mathcal{C}^{1}(I,V_{\Lambda}(\Omega))\right]_{B},$
\begin{equation}
\forall\varphi\in\mathscr{D}(I\times\Omega),\left\langle \left[u\right]_{\mathscr{D}(I\times\Omega)},\varphi\right\rangle =sh\int\int u(t,x)\varphi^{\ast}(t,x)dxdt.\label{pipa}
\end{equation}

\end{prop}
Finally, let us also notice that the proof of Theorem \ref{bello}
can be modified to prove the following result: 
\begin{prop}
If $W(E)$ is an internal space such that $\mathscr{D}^{\ast}(E)\subset W(E)\subset L^{2}(E)^{\ast}$,
then every distribution $\left[v\right]_{\mathscr{D}}$ has a representative
$u\in W(E)\cap\left[v\right]_{\mathscr{D}}$. Namely, the map 
\[
\Phi:[W(E)]_{B}\rightarrow\mathscr{D}^{\prime}(E)
\]
defined by 
\[
\Phi(u)=\left[u\right]_{\mathscr{D}}
\]
is surjective.\end{prop}
\begin{proof}
We can argue as in the proof of Thm \ref{bello}, by substituting
$L^{2}(E)$ with $\mathscr{D}(E)$. This is possible since $\mathscr{D}(E)$
is dense in $L^{2}(E)$ (and so, in particular, $W(E)$ is dense in
$L^{2}(E)$), and the density property was the only condition needed
to prove the surjectivity of the embedding. 
\end{proof}
In the following sections we want to study problems such as equation
(\ref{pura}) in the context of ultrafunctions. To do so we will need
to restrict to the following family of operators: 
\begin{defn}
\label{wc} We say that an operator 
\[
\mathcal{A}:V(\Omega)\rightarrow V^{\prime}(\Omega)
\]
is weakly continuous if, $\forall u,v\in\left[V_{\Lambda}(\Omega)\right]_{B},\ \forall\varphi\in\mathscr{D}(\Omega)$,
we have that if 
\[
\int u\varphi^{\ast}\ dx\sim\int v\varphi^{\ast}\ dx
\]
then 
\[
\int\mathcal{A}^{\ast}(u)\ \varphi^{\ast}\ dx\sim\int\mathcal{A}^{\ast}(v)\ \varphi^{\ast}\ dx.
\]

\end{defn}
For our purposes, the important property of weakly continuous operators
is that if 
\[
\mathcal{A}:V(\Omega)\rightarrow V^{\prime}(\Omega)
\]
is weakly continuous then it can be extended to an operator 
\[
\left[\mathcal{A}\right]_{\mathfrak{\mathscr{D}}}:\mathfrak{\mathscr{D}}^{\prime}(\Omega)\rightarrow\mathfrak{\mathscr{D}}^{\prime}(\Omega)
\]
by setting 
\[
\left[\mathcal{A}\right]_{\mathscr{D}}\left(\left[u\right]_{\mathscr{D}}\right)=\left[\mathcal{A}\left(w\right)\right]_{\mathscr{D}},
\]
where $w\in\left[u\right]_{\mathscr{D}}\cap V(\Omega).$ In the following,
with some abuse of notation we will write $\left[\mathcal{A}\left(u\right)\right]_{\mathscr{D}}$
instead of $\left[\mathcal{A}\right]_{\mathfrak{\mathscr{D}}}\left(\left[u\right]_{\mathscr{D}}\right).$ 
\begin{rem}
Definition \ref{wc} can be reformulated in the classical language
as follows: $\mathcal{A}$ is weakly continuous if for every weakly
convergent sequence $u_{n}$ in $\mathfrak{\mathscr{D}}^{\prime}(\Omega)$
the sequence $\mathcal{A}\left(u_{n}\right)$ is weakly convergent
in $\mathfrak{\mathscr{D}}^{\prime}(\Omega)$. 
\end{rem}

\section{Generalized Ultrafunction Solutions (GUS)\label{gus}}

In this section we will show that an evolution equation such as equation
(\ref{pura}) has Generalized Ultrafunction Solutions (GUS) under
very general assumptions on $\mathcal{A}$, and we will show the relationships
of GUS with strong and weak solutions. However, before doing this,
we think that it is helpful to give the feeling of the notion of GUS
for stationary problems. This will be done in Section \ref{gussp}
providing a simple typical example. We refer to \cite{ultra}, \cite{belu2012}
and \cite{milano} for other examples.

\subsection{Generalized Ultrafunction Solutions for stationary problems\label{gussp}}

A typical stationary problem in PDE can be formulated ad follows:
\[
\text{\textit{Find}\ \ \ }u\in V(\Omega)\ \ \ \text{\textit{such that}}
\]
\begin{equation}
\mathcal{A}(u)=f,\label{ps}
\end{equation}
where $V(\Omega)\subseteq L^{2}(\Omega)$ is a vector space and $\mathcal{A}:V(\Omega)\rightarrow V^{\prime}(\Omega)$
is a differential operator and $f\in L^{2}(\Omega)$.

The \textquotedbl{}typical\textquotedbl{} formulation of this problem
in the framework of ultrafunctions is the following one: 
\[
\text{\textit{Find} \ }u\in V_{\Lambda}(\Omega)\ \ \text{\textit{such that}}
\]
\begin{equation}
\widetilde{\mathcal{A}}(u)=\widetilde{f}.\label{P}
\end{equation}
In particular, if $\mathcal{A}:V(\Omega)\rightarrow L^{2}(\Omega)$
and $f\in L^{2}(\Omega),$ the above problem can be formulated in
the following equivalent \textquotedbl{}weak form\textquotedbl{}:

\[
\text{\textit{Find} \ }u\in V_{\Lambda}(\Omega)\ \ \text{\textit{such that}}
\]
\begin{equation}
\forall\varphi\in V_{\Lambda}(\Omega),\ \int_{\Omega^{\ast}}^{\ast}\mathcal{A}^{\ast}(u)\varphi dx=\int_{\Omega^{\ast}}^{\ast}f^{\ast}\varphi dx.
\end{equation}

Such an ultrafunction $u$ will be called a \textbf{GUS} of Problem
(\ref{P}).

Usually, it is possible to find a classical solution for problems
of the type (\ref{ps}) if there are a priory bounds, but the existence
of a priori bounds is not sufficient to guarantee the existence of
solutions in $V(\Omega).$ On the contrary, the existence of a priori
bounds is sufficient to find a GUS in $V_{\Lambda}(\Omega)$ (as we
are going to show).

Following the general strategy to find a GUS for Problem (\ref{P}),
we start by solving the following approximate problems for every $\lambda$
in a qualified set : 
\[
\text{\textit{Find} }u_{\lambda}\in V_{\lambda}(\Omega)\text{ \ \textit{such\ that}}
\]
\[
\forall\varphi\in V_{\lambda}(\Omega),\ \int_{\Omega}\mathcal{A}(u_{\lambda})\varphi dx=\int_{\Omega}f\varphi dx.
\]

A priori bounds in each space $V_{\lambda}(\Omega)$ are sufficient
to guarantee the existence of solutions. The next step consists in
taking the $\Lambda$-limit. Clearly, this strategy can be applied
to a very large class of problems. Let us consider a typical example
in details: 
\begin{thm}
\label{nl}Let $\mathcal{A}:\,V(\Omega)\rightarrow V^{\prime}(\Omega)$
be a hemicontinuous\footnote{An operator between Banach spaces is called \emph{hemicontinuous}
if its restriction to finite dimensional subspaces is continuous.} operator such that for every finite dimensional space $V_{\lambda}\subset V(\Omega)$
there exists $R_{\lambda}\in\mathbb{R}$ such that 
\begin{equation}
\text{if}\ u\in V_{\lambda}\text{and}\ \left\Vert u\right\Vert _{\sharp}=R_{\lambda}\ \text{then}\ \left\langle \mathcal{A}(u),u\right\rangle >0,\label{gilda}
\end{equation}
where $\left\Vert \cdot\right\Vert _{\sharp}$ is any norm in $V(\Omega).$
Then the equation (\ref{P}) has at least one solution $u_{\Lambda}\in V_{\Lambda}(\Omega).$ \end{thm}
\begin{proof}
If we set 
\[
B_{\lambda}=\left\{ u\in V_{\lambda}|\ \left\Vert u\right\Vert _{\sharp}\leq R_{\lambda}\right\} 
\]
and if $\mathcal{A}_{\lambda}:V_{\lambda}\rightarrow V_{\lambda}$
is the operator defined by the following relation: 
\[
\forall v\in V_{\lambda},\ \left\langle \mathcal{A}_{\lambda}(u),v\right\rangle =\left\langle \mathcal{A}(u),v\right\rangle 
\]
then it follows from the hypothesis (\ref{gilda}) that $\deg(\mathcal{A}_{\lambda},B_{\lambda},0)=1,\ $where
$\deg(\cdot,\cdot,\cdot)$ denotes the topological degree (see e.g.
\cite{AmMa2003}). Hence, $\forall\lambda\in\mathfrak{L}$, 
\[
\exists u\in V_{\lambda},\forall v\in V_{\lambda},\ \left\langle \mathcal{A}_{\lambda}(u),v\right\rangle =0.
\]
Taking the $\Lambda$-limit of the net $(u_{\lambda})$ we get a GUS
$u_{\Lambda}\in V_{\Lambda}(\Omega)$ of equation (\ref{P}).\end{proof}
\begin{example}
Let $\Omega$ be an open bounded set in $\mathbb{R}^{N}$ and let
\[
a(\cdot,\cdot,\cdot):\mathbb{R}^{N}\times\mathbb{R}\times\overline{\Omega}\rightarrow\mathbb{R}^{N},\ b(\cdot,\cdot,\cdot):\mathbb{R}^{N}\times\mathbb{R}\times\overline{\Omega}\rightarrow\mathbb{R}
\]
be continuous functions such that $\forall\xi\in\mathbb{R}^{N},\forall s\in\mathbb{R},\forall x\in\overline{\Omega}$
we have 
\begin{equation}
a(\xi,s,x)\cdot\xi+b(\xi,s,x)\geq\nu\left(|\xi|\right),\label{lalla}
\end{equation}
where $\nu$ is a function (not necessarely negative) such that 
\begin{equation}
\nu\left(t\right)\rightarrow+\infty\ \text{for\ }t\rightarrow+\infty.\label{lalla1}
\end{equation}
We consider the following problem: 
\[
\text{\textit{Find\ }\ }u\in\mathcal{C}^{2}(\Omega)\cap\mathcal{C}_{0}(\overline{\Omega})\ \ \text{\textit{s.t.}}
\]
\begin{equation}
\nabla\cdot a(\nabla u,u,x)=b(\nabla u,u,x).\label{5th}
\end{equation}
In the framework of ultrafunctions this problem becomes 
\[
\text{\textit{Find} \ }u\in V(\Omega):=\left[\mathcal{C}^{2}(\Omega)\cap\mathcal{C}_{0}(\overline{\Omega})\right]_{\Lambda}\ \ \text{\textit{such that}}
\]
\[
\forall\varphi\in V(\Omega),\ \int_{\Omega}\nabla\cdot a(\nabla u,u,x)\ \varphi\ dx=\int_{\Omega}b(\nabla u,u,x)\varphi dx.
\]
If we set 
\[
\mathcal{A}(u)=-\nabla\cdot a(\nabla u,u,x)+b(\nabla u,u,x)
\]
it is not difficult to check that conditions (\ref{lalla}) and (\ref{lalla1})
are sufficient to guarantee the assumptions of Thm.~\ref{nl}. Hence
we have the existence of a ultrafunction solution of problem (\ref{5th}).
Problem (\ref{5th}) covers well known situations such as the case
in which $\mathcal{A}$ is a maximal monotone operator, but also very
pathological cases. E.g., by taking 
\[
a(\nabla u,u,x)=(|\nabla u|^{p-1}-\nabla u);\ b(\nabla u,u,x)=f(x),
\]
we get the equation 
\[
\Delta_{p}u-\Delta u=f.
\]
Since 
\[
\int_{\Omega}\left(-\Delta_{p}u+\Delta u\right)\ u\ dx=\left\Vert u\right\Vert _{W_{0}^{1,p}}^{p}-\left\Vert u\right\Vert _{H_{0}^{1}}^{2},
\]
it is easy to check that we have a priori bounds (but not the convergence)
in $W_{0}^{1,p}(\Omega)$. Therefore we have GUS, and it might be
interesting to study the kind of regularity of these solutions. 
\end{example}

\subsection{Strong and weak solutions of evolution problems\label{sub:Strong-and-weak}}

As usual, let 
\[
\mathcal{A}:V(\Omega)\rightarrow V^{\prime}(\Omega)
\]
be a differential operator.

We are interested in the following Cauchy problem for $t\in I:=\left[0,T\right)$:
find $u$ such that 
\begin{equation}
\left\{ \begin{array}{c}
\partial_{t}u=\mathcal{A}(u);\\
\\
u\left(0\right)=u_{0}.
\end{array}\right.\label{cp}
\end{equation}
A solution $u=u(t,x)$ of problem (\ref{cp}) is called a \textbf{strong
solution} if 
\[
u\in C^{0}(I,V(\Omega))\cap C^{1}(I,V^{\prime}(\Omega)).
\]
It is well known that many problems of type (\ref{cp}) do not have
strong solutions even if the initial data is smooth (for example Burgers'
equation \ref{BE}). This is the reason why the notion of weak solution
becomes necessary. If $\mathcal{A}$ is a linear operator and $\mathcal{A}\left(\mathfrak{\mathcal{\mathscr{D}}}(\Omega)\right)\subset\mathfrak{\mathcal{\mathscr{D}}}'(\Omega)$,
classically a distribution $T\in V^{\prime}(I\times\mathbb{R})$ is
called a weak solution of problem (\ref{cp}) if 
\[
\forall\varphi\in\mathfrak{\mathscr{D}}(I\times\mathbb{R}),\ -\left\langle T,\partial_{t}\varphi\right\rangle +\int_{\Omega}u_{0}(x)\ \varphi(0,x)dx=\left\langle T,\mathcal{A}^{\dag}\varphi\right\rangle ,
\]
where $\mathcal{A}^{\dag}$ is the adjoint of $\mathcal{A}$.

If $\mathcal{A}$ is not linear there is not a general definition
of weak solution. For example, if you consider Burgers' equation,
a function $w\in L_{loc}^{1}(I\times\Omega)$ is considered a weak
solution if 
\[
\forall\varphi\in\mathfrak{\mathscr{D}}(I\times\Omega),\ -\int\int w\partial_{t}\varphi\ dxdt-\int_{\Omega}u_{0}(x)\ \varphi(0,x)dx+\frac{1}{2}\int\int w^{2}\partial_{x}\varphi\ dxdt=0.
\]

However, if we use the notion of generalized distribution developed
in section \ref{distri} we can give a definition of weak solution
for problems involving weakly continuous operators that generalizes
the classical one for linear operators: 
\begin{defn}
\label{DS} Let $\mathcal{A}:W\rightarrow\mathscr{D}^{\prime}$ be
weakly continuous. We say that $u\in W$ is a weak solution of Problem
(\ref{cp}) if the following condition is fulfilled: $\forall\varphi\in\mathfrak{\mathscr{D}}(I\times\Omega)$
\[
\int u(t,x)\varphi_{t}(t,x)\,dxdt-\int u(0,x)\varphi(0,x)dx=\langle A(u),\varphi\rangle.
\]

\end{defn}
From the theory developed in Section \ref{distri}, the notion of
weak solution given by Definition \ref{DS} can be written in nonstandard
terms as follows: $\left[w\right]_{\mathscr{D}}$ is a weak solution
of Problem (\ref{cp}) if

\[
\left\{ \begin{array}{c}
w\in\left[C^{1}(I,V(\Omega))^{\ast}\right]_{B};\\
\\
\forall\varphi\in\mathscr{D}(I\times\Omega),\ \int_{0}^{T}\int_{\Omega}\partial_{t}w\varphi^{\ast}dxdt+\int_{0}^{T}\mathcal{A}\left(w\right)\varphi^{\ast}dt\sim0;\\
\\
w\left(0,x\right)=u_{0}(x).
\end{array}\right.
\]

By the above equations, any strong solution is a weak solution, but
the converse is not true. A very large class of problems (such as
\ref{BE}) which do not have strong solutions have weak solutions,
or even only distributional solutions. Unfortunately, there are problems
which do not have even weak (or distributional) solutions, and worst
than that there are problems (such as \ref{BE}) which have more than
one weak solution, namely the uniqueness of the Cauchy problem is
violated, and hence the physical meaning of the problem is lost. This
is why we think that it is worthwhile to investigate these kind of
problems in the framework of generalized solutions in the world of
ultrafunctions.

\subsection{Generalized Ultrafunction Solutions and their first properties}

In Section \ref{gussp} we gave the definition of GUS for stationary
problems. The definition of GUS for evolution problems is analogous: 
\begin{defn}
\label{def:pippa}An ultrafunction $u\in\mathcal{C}^{1}(I^{\ast},V_{\Lambda}(\Omega)),$
is called a Generalized Ultrafunction Solution (GUS) of problem (\ref{cp})
if $\forall v\in V_{\Lambda}(\Omega),\ $ 
\begin{equation}
\left\{ \begin{array}{c}
\int\partial_{t}uv\ dx=\int\mathcal{A}^{\mathcal{\ast}}(u)v\ dx;\\
\\
u\left(0,x\right)=u_{0}\left(x\right).
\end{array}\right.\label{q}
\end{equation}

\end{defn}
Problem (\ref{q}) can be rewritten as follows: 
\[
\left\{ \begin{array}{c}
u\in\mathcal{C}^{1}(I^{\ast},V_{\Lambda});\notag\\
\\
\partial_{t}u=P_{\Lambda}\mathcal{A}^{\mathcal{\ast}}(u);\label{esse}\\
\\
u\left(0,x\right)=u_{0}\left(x\right),\notag
\end{array}\right.
\]
where $P_{\Lambda}$ is the orthogonal projection. The main Theorem
of this section states that problem (\ref{cp}) locally has a GUS.
As for the ordinary differential equations in finite dimensional spaces,
this solution is defined for an interval of time which depends on
the initial data. 
\begin{thm}
\label{TT}Let $\mathcal{A}|_{V_{\lambda}(\Omega)}$ be locally Lischitz
continuous $\forall\lambda\in\mathfrak{L}$; then there exists a number
$T_{\Lambda}(u_{0})\in\left(0,T\right]_{\mathbb{R}^{\ast}}$ such
that problem (\ref{cp}) has a unique GUS $u_{\Lambda}$ in $\left[0,T_{\Lambda}(u_{0})\right)_{\mathbb{R}^{\ast}}.$ \end{thm}
\begin{proof}
For every $\lambda\in\mathfrak{L}$ let us consider the approximate
problem 
\begin{equation}
\left\{ \begin{array}{c}
u\in C^{1}(I,V_{\lambda}(\Omega))\ \ and\ \ \forall v\in V_{\lambda}(\Omega);\\
\\
\int_{\Omega}\partial_{t}u(t,x)\ v(x)dx=\int_{\Omega}\mathcal{A}(u(t,x))\ v(x)dx;\\
\\
u_{\lambda}\left(0\right)=\int_{\Omega}u_{0}(x)\ v(x)dx.
\end{array}\right.\label{gonerilla}
\end{equation}
It is immediate to check that this problem is equivalent to the following
one 
\begin{equation}
\left\{ \begin{array}{c}
u\in C^{1}(I,V_{\lambda}(\Omega));\\
\\
\partial_{t}u(t,x)=P_{\lambda}\mathcal{A}(u(t,x));\\
\\
u_{\lambda}\left(0\right)=P_{\lambda}u_{0},
\end{array}\right.\label{regana}
\end{equation}
where the \textquotedbl{}projection\textquotedbl{} $P_{\lambda}:L^{2}(\Omega)\rightarrow V_{\lambda}(\Omega)$
is defined by 
\begin{equation}
\int_{\Omega}P_{\lambda}w(x)v(x)dx=\left\langle w,v\right\rangle ,\ \forall v\in V_{\lambda}(\Omega).\label{PL}
\end{equation}
The Cauchy problem (\ref{regana}) is well posed since $V_{\lambda}(\Omega)$
is a finite dimensional vector space and $P_{\lambda}\circ$ $\mathcal{A}$
is locally Lipschitz continuous on $V_{\lambda}$. Then there exists
a number $T_{\lambda}(u_{0})\in\left(0,T\right]_{\mathbb{R}}$ such
that problem (\ref{regana}) has a unique solution in $\left[0,T_{\lambda}(u_{0})\right)_{\mathbb{R}}.$
Taking the $\Lambda$-limit, we get the conclusion. \end{proof}
\begin{defn}
We will refer to a solution $u_{\Lambda}$ given as in Theorem \ref{TT}
as to a local GUS. 
\end{defn}
Clearly the GUS is a global solution (namely a function defined for
every $t\in\left[0,T\right)$) if $T_{\lambda}(u_{0})$ is equal to
$T$. In concrete applications, the existence of a global solution
usually is a consequence of the existence of a coercive integral of
motion. In fact, we have the following corollary: 
\begin{cor}
\label{coco}Let the assumptions of Thm.~\ref{TT} hold. Moreover,
let us assume that there exists a function $I:V(\Omega)\rightarrow\mathbb{R}$
such that if $u(t)$ is a local GUS in $\left[0,T_{\lambda}\right)$,
then 
\begin{equation}
\partial_{t}I^{\ast}\left(u(t)\right)\leq0\label{vend}
\end{equation}
(or, more in general, that $I^{\ast}\left(u(t)\right)$ is not increasing)
and such that $\forall\lambda\in\mathfrak{L},\ I|_{V_{\lambda}(\Omega)}$
is coercive (namely if $u_{n}\in V_{\lambda}(\Omega)$ and $\left\Vert u_{n}\right\Vert \rightarrow\infty$
then $I\left(u_{n}\right)\rightarrow\infty$). Then $u(t)$ can be
extended to the full interval $\left[0,T\right).$ \end{cor}
\begin{proof}
By our assumptions, there is a qualified set $Q$ such that $\forall\lambda\in Q$,
if $u_{\lambda}(t)$ is defined in $\left[0,T_{\lambda}\right),$
then 
\begin{equation}
\partial_{t}I\left(u_{\lambda}(t)\right)\leq0\label{vend1}
\end{equation}
since otherwise the inequality (\ref{vend}) would be violated. By
(\ref{vend1}) and the coercivity of $I|_{V_{\lambda}(\Omega)}$ we
have that $T_{\lambda}(u_{0})=T.$ Hence also $u(t)$ is defined in
the full interval $\left[0,T\right).$ 
\end{proof}

\subsection{GUS, weak and strong solutions}

We now investigate the relations between GUS, weak solutions and strong
solutions. 
\begin{thm}
\label{mina}Let $u\in C^{1}(I^{\ast},V_{\Lambda}(\Omega))$ be a
GUS of Problem (\ref{cp}), and let us assume that $\mathcal{A}$
is weakly continuous. Then 
\begin{enumerate}
\item \label{enu:if-then-the 1}if 
\[
u\in\left[\mathcal{C}^{1}(I^{\ast},V_{\Lambda}(\Omega))\right]_{B}
\]
then the distribution $\left[u\right]_{\mathscr{D}}$ is a weak solution
of Problem (\ref{cp}); 
\item moreover, if 
\[
w\in\left[u\right]_{\mathscr{D}}\cap\mathcal{C}^{1}(I,V(\Omega))
\]
then $w$ is a strong solution of Problem (\ref{cp}). 
\end{enumerate}
\end{thm}
\begin{proof}
(1) In order to simplify the notations, in this proof we will write
$\int$ instead of $\int^{\ast}$. Since $u$ is a GUS, then for any
$\varphi\in\mathscr{D}\left(I\times\Omega\right)\subset\mathcal{C}_{B}^{\infty}(I^{\ast},V_{\Lambda}(\Omega))\ $
(we identify $\varphi$ and $\varphi^{\ast}$) we have that 
\[
\int_{0}^{T}\int_{\Omega^{\ast}}\partial_{t}u\varphi\ dx\ dt=\int_{0}^{T}\int_{\Omega^{\ast}}\mathcal{A^{\ast}}(u)\varphi dx\ dt.
\]
Integrating in $t$, we get 
\[
\int_{0}^{T}\int_{\Omega^{\ast}}u(t,x)\ \partial_{t}\varphi dx\ dt-\int_{\Omega}u_{0}(x)\varphi(0,x)dx+\int_{0}^{T}\int_{\Omega^{\ast}}\mathcal{A}^{\ast}(u(t,x))\varphi dx\ dt=0.
\]
By the definition of $\left[u\right]_{\mathscr{D}}$, and as $\mathcal{A}$
is weakly continuous, we have that 
\[
\int_{0}^{T}\int_{\Omega^{\ast}}u(t,x)\ \partial_{t}\varphi dx\ dt\sim\int_{0}^{T}\int_{\Omega}[u]_{\mathscr{D}}(t,x)\ \partial_{t}\varphi dx\ dt,
\]
\[
\int_{\Omega^{\ast}}u_{0}(x)\varphi(0,x)dx\sim\int_{\Omega}\left([u]_{\mathscr{D}}\right)_{0}(x)\varphi(0,x)dx,
\]
\[
\int_{0}^{T}\int_{\Omega^{\ast}}\mathcal{A^{\ast}}(u(t,x))\varphi dx\ dt\sim\int_{0}^{T}\int_{\Omega}\mathcal{A}([u]_{\mathscr{D}}(t,x))\varphi dx\ dt.
\]
Henceforth 
\[
\int_{0}^{T}\int_{\Omega}[u]_{\mathscr{D}}(t,x)\ \partial_{t}\varphi dx\ dt-\int_{\Omega}\left([u]_{\mathscr{D}}\right)_{0}(x)\varphi(0,x)dx+\int_{0}^{T}\int_{\Omega}\mathcal{A}([u]_{\mathscr{D}}(t,x))\varphi dx\ dt\sim0.
\]
Since all three terms in the left hand side of the above equation
are real numbers, we have that their sum is a real number, and so
\[
\int_{0}^{T}\int_{\Omega}[u]_{\mathscr{D}}(t,x)\ \partial_{t}\varphi dx\ dt-\int_{\Omega}\left([u]_{\mathscr{D}}\right)_{0}(x)\varphi(0,x)dx+\int_{0}^{T}\int_{\Omega}\mathcal{A}([u]_{\mathscr{D}}(t,x))\varphi dx\ dt=0,
\]
namely $[u]_{\mathscr{D}}$ is a weak solution of Problem (\ref{cp}).

(2) If there exists $w\in\left[u\right]_{\mathscr{D}}\cap\mathcal{C}^{1}(I,V(\Omega))$
then $u\in\left[\mathcal{C}^{1}(I^{\ast},V_{\Lambda}(\Omega))\right]_{B}$,
so from $(\ref{enu:if-then-the 1})$ we get that $w$ is a weak solution
of Problem (\ref{cp}). Moreover, $w\in\mathcal{C}^{1}(I,V(\Omega))\subseteq\mathcal{C}^{0}(I,V(\Omega))\cap\mathcal{C}^{1}(I,V^{\prime}(\Omega))$,
and hence $w$ is a strong solution.
\end{proof}
Usually, if problem (\ref{cp}) has a strong solution $w$, it is
unique and it coincides with the GUS $u$ in the sense that $\left[w^{\ast}\right]_{\mathscr{D}}=\left[u\right]_{\mathscr{D}}$
and in many cases we have also that 
\begin{equation}
\left\Vert u-w^{\ast}\right\Vert \sim0.\label{fagiolino}
\end{equation}
If problem (\ref{cp}) does not have a strong solution but only weak
solutions, often they are not unique. Thus the GUS selects one weak
solution among them. 

Now suppose that $w\in L_{loc}^{1}$ is a weak solution such that
$\left[u\right]_{\mathscr{D}}=\left[w^{\ast}\right]_{\mathscr{D}}$
but (\ref{fagiolino}) does not hold. If we \foreignlanguage{english}{set}
\[
\psi=u-w^{\ast}
\]
then $\left\Vert \psi\right\Vert $ is not an infinitesimal and $\psi$
carries some information which is not contained in $w.$ Since $u$
and $w$ define the same distribution, $\left[\psi\right]_{\mathscr{D}}=0,$
i.e. 
\[
\forall\varphi\in\mathfrak{\mathscr{D},\ }\int\psi\varphi^{\ast}\ dx=0.
\]
So the information contained in $\psi$ cannot be contained in a distribution.
Nevertheless this information might be physically relevant. In Section
\ref{sub:The-microscopic-part}, we will see one example of this fact.

\subsection{First example:\ the nonlinear Schroedinger equation\label{se}}

Let us consider the following nonlinear Schroedinger equation in $\mathbb{R}^{N}:$
\begin{equation}
i\partial_{t}u=-\frac{1}{2}\Delta u+V(x)u-|u|^{p-2}u;\ p>2,\label{NSE}
\end{equation}
where, for simplicity, we suppose that $V(x)\in\mathcal{C}^{1}(\mathbb{R}^{N})$
is a smooth bounded potential. A suitable space for this problem is
\[
V(\mathbb{R}^{N})=H^{2}(\mathbb{R}^{N})\cap L^{p}(\mathbb{R}^{N})\cap\mathcal{C}(\mathbb{R}^{N}).
\]

In fact, if $u\in V(\mathbb{R}^{N}),$ then the energy 
\begin{equation}
E(u)=\int\left[\frac{1}{2}\left\vert \nabla u\right\vert ^{2}+V(x)\left\vert u\right\vert ^{2}+\frac{2}{p}|u|^{p}\right]dx\label{senergia}
\end{equation}
is well defined; moreover, if $u\in V(\mathbb{R}^{N})$ we have that
\[
-\frac{1}{2}\Delta u+V(x)u-|u|^{p-2}u\in V^{\prime}(\mathbb{R}^{N}),
\]
so the problem is well-posed in the sense of ultrafunctions (see Def.
\ref{def:pippa}). It is well known, (see e.g. \cite{Ca03}) that
if $p<2+\frac{4}{N}$ then the Cauchy problem (\ref{NSE}) (with initial
data in $V(\mathbb{R}^{N})$) is well posed, and there exists a strong
solution 
\[
u\in C^{0}(I,V(\mathbb{R}^{N}))\cap C^{1}(I,V^{\prime}(\mathbb{R}^{N})).
\]
On the contrary, if $p\geq2+\frac{4}{N},$ the solutions, for suitable
initial data, blows up in a finite time. So in this case weak solutions
do not exist. Nevertheless, we have GUS: 
\begin{thm}
The Cauchy problem relative to equation (\ref{NSE}) with initial
data $u_{0}\in V_{\Lambda}(\mathbb{R}^{N})$ has a unique GUS $u\in\mathcal{C}^{1}(I,V_{\Lambda}(\mathbb{R}^{N}));$
moreover, the energy (\ref{senergia}) and the $L^{2}$-norm are preserved
along this solution. \end{thm}
\begin{proof}
Let us consider the functional 
\[
I(u)=\int|u|^{2}dx.
\]

On every approximating space $V_{\lambda}(\Omega)$ we have that 
\[
\frac{d}{dt}\int|u|^{2}dx=\int\frac{d}{dt}|u|^{2}dx=2Re\int\left(u,\frac{d}{dt}u\right)=0,
\]
therefore $I^{\ast}$ (namely, the $L^{2}$-norm) is constant on GUS.
A similar direct computation can be used to prove that also the energy
is constant on GUS. Moreover, it is easily seen that $\forall\lambda\in\mathfrak{L},\ I|_{V_{\lambda}(\Omega)}$
is coercive. Since also the other hypotheses of Theorem \ref{TT}
are verified, we can apply Corollary \ref{coco} to get the existence
and uniqueness of the GUS. 
\end{proof}
Now it is interesting to know what these solutions look like, and
if they have any reasonable meaning from the physical or the mathematical
point of view. For example, when $p<2+\frac{4}{N}$ the dynamics given
by equation (\ref{NSE}), for suitale initial data, produces solitons
(see e.g. \cite{befolib} or \cite{milan}); so we conjecture that
in the case $p\geq2+\frac{4}{N}$ solitons with infinitesimal radius
will appear at the concentration points and that they will behave
as pointwise particles which follow the Newtonian Dynamics.

\subsection{Second example: the nonlinear wave equation}

Let us consider the following Cauchy problem relative to a nonlinear
wave equation in a bounded open set $\Omega\subset\mathbb{R}^{N}:$
\begin{equation}
\left\{ \begin{array}{ccc}
\square\psi+|\psi|^{p-2}\psi & = & 0\,\mbox{\,\ in\,\,}I\times\Omega;\\
\\
\psi & = & 0\,\mbox{\,\ on}\,\,I\times\partial\Omega;\\
\\
\psi(0,x) & = & \psi_{0}(x),
\end{array}\right.\label{cordelia}
\end{equation}
where $\square=\partial_{t}^{2}-\Delta,\ p>2\,,I=[0,T)$. In order
to formulate this problem in the form (\ref{cp}), we reduce it to
a system of first order equations (Hamiltonian formulation): 
\[
\left\{ \begin{array}{c}
\partial_{t}\psi=\phi;\\
\\
\partial_{t}\phi=\Delta\psi-|\psi|^{p-2}\psi.
\end{array}\right.
\]
If we set 
\[
u=\left[\begin{array}{c}
\psi\\
\phi
\end{array}\right];\ \mathcal{A}(u)=\left[\begin{array}{c}
\phi\\
\Delta\psi-|\psi|^{p-2}\psi
\end{array}\right],
\]
then problem (\ref{cordelia}) reduces to a particular case of problem
(\ref{cp}).

A suitable space for this problem is 
\[
V(\Omega)=\left[\mathcal{C}^{2}(\Omega)\cap\mathcal{C}_{0}(\overline{\Omega})\right]\times\mathcal{C}(\Omega).
\]
If $u\in V(\Omega),$ the energy 
\begin{equation}
E(u)=\int_{\Omega}\left[\frac{1}{2}\left\vert \phi\right\vert ^{2}+\frac{1}{2}\left\vert \nabla\psi\right\vert ^{2}+\frac{1}{p}|\psi|^{p}\right]dx\label{energia2}
\end{equation}
is well defined.

It is well known, (see e.g. \cite{JLL}) that problem (\ref{cordelia})
has a weak solution; however, it is possible to prove the global uniqueness
of such a solution only if $p<\frac{N}{N-2}$ (any $p$ if $N=1,2$).

On the contrary, in the framework of ultrafunctions we have the following
result: 
\begin{thm}
The Cauchy problem relative to equation (\ref{cordelia}) with initial
data $u_{0}\in V_{\Lambda}(\Omega)$ has a unique solution $u\in\mathcal{C}^{1}(I^{\ast},V_{\Lambda}(\Omega));$
moreover, the energy (\ref{energia2}) is preserved along this solution. \end{thm}
\begin{proof}
We have only to apply Theorem \ref{TT} and Corollary \ref{coco},
where we set 
\[
I(u):=E(u)=\int_{\Omega}\left[\frac{1}{2}\left\vert \phi\right\vert ^{2}+\frac{1}{2}\left\vert \nabla\psi\right\vert ^{2}+\frac{1}{p}|\psi|^{p}\right]dx.\qedhere
\]

\end{proof}

\section{The Burgers' equation\label{tbe}}

\subsection{Preliminary remarks}

In section \ref{se} we have shown two examples which show that: 
\begin{itemize}
\item equations which do not have weak solutions usually have a unique GUS; 
\item equations which have more than a weak solution have a unique GUS. 
\end{itemize}
So ultrafunctions seem to be a good tool to study the phenomena modelled
by these equations. At this point we think that the main question
is to know what the GUS look like and if they are suitable to represent
properly the phenomena described by such equations from the point
of view of Physics. Of course this question might not have a unique
answer: probably there are phenomena which are well represented by
GUS and others which are not. In any case, it is worthwhile to investigate
this issue relatively to the main equations of Mathematical Physics
such as (\ref{NSE}), (\ref{cordelia}), Euler equations, Navier-Stokes
equations and so on.

We have decided to start this program with the (nonviscous) Burgers'
equation

\[
\frac{\partial u}{\partial t}+u\frac{\partial u}{\partial x}=0,
\]
since it presents the following peculiarities: 
\begin{itemize}
\item it is one of the (formally) simplest nonlinear PDE; 
\item it does not have a unique weak solution, but there is a physical criterium
to determine the solution which has physical meaning (namely the entropy
solution); 
\item many solutions can be written explicitly, and this helps to confront
classical and ultrafunction solutions. 
\end{itemize}
We recall that an other interesting approach to Burgers' equation
by means of generalized functions (in the Colombeau sense) has been
devoloped by Biagioni and Oberguggenberger in \cite{biagioni-MO}.

\subsection{Properties of the GUS of Burgers' equations}

The first property of Burgers' equation (\ref{BE}) that we prove
is that its smooth solutions with compact support have infinitely
many integrals of motion: 
\begin{prop}
\label{tina}Let $G(u)$ be a differentiable function, $G\in\mathcal{C}^{1}(\mathbb{R}),$
$G(0)=0,$ and let $u(t,x)$ be a smooth solution of (\ref{BE}) with
compact support. Then 
\[
I(u)=\int G(u(t,x))dx
\]
is a constant of motion of (\ref{BE}) (provided that the integral
converges). \end{prop}
\begin{proof}
The proof of this fact is known, we include it here only for the sake
of completeness. Multiplying both sides of equation (\ref{BE}) by
$G^{\prime}(u),$ we get the equation 
\[
G^{\prime}(u)\partial_{t}u+G^{\prime}(u)u\partial_{x}u=0,
\]
which gives 
\[
\partial_{t}G(u)+\partial_{x}H(u)=0,
\]
where 
\begin{equation}
H(u)=\int_{0}^{u}sG^{\prime}(s)ds.\label{lilla}
\end{equation}
Since $u$ has compact support, we have that $-\int\partial_{x}H(u)dx=0,$
and hence 
\[
\partial_{t}\int G(u)dx=-\int\partial_{x}H(u)dx=0.\qedhere
\]

\end{proof}
Let us notice that Proposition \ref{tina} would hold also if we do
not assume that $u$ has a compact support, provided that it decays
sufficiently fast. 

In the literature, any function $G$ as in the above theorem is called
entropy and $H$ is called entropy flux (see e.g. \cite{Bianchini,MinEnt}),
since in some interpretation of this equation $G$ corresponds (up
to a sign) the the physical entropy. But this is not the only possible
interpretation. 

If we interpret (\ref{BE}) as a simplification of the Euler equation,
the unknown $u$ is the velocity; then, for $G(u)=u\ $ and $\ G(u)=\frac{1}{2}u^{2},$
we have the following constants of motion: the \textbf{momentum} 
\[
P(u)=\int udx
\]
and the \textbf{energy} 
\[
E(u)=\frac{1}{2}\int u^{2}dx.
\]

However, in general the solutions of Burgers' equation are not smooth;
in fact, if the initial data $u_{0}(x)$ is a smooth function with
compact support, the solution develops singularities. Hence we must
consider weak solutions which, in this case, are solutions of the
following equation in weak form: $w\in L_{loc}^{1}(I\times\Omega)$,
and $\forall\varphi\in\mathscr{D}(I\times\Omega)$

\begin{multline}
\int_{0}^{T}\int_{\Omega}w(t,x)\partial_{t}\varphi(t,x)\ dxdt-\int_{\Omega}u_{0}(x)\varphi(0,x)dx+\\
\frac{1}{2}\int_{0}^{T}\int_{\Omega}w(t,x)^{2}\partial_{x}\varphi(t,x)\ dxdt=0.\label{zanna}
\end{multline}

Nevertheless, the momentum and the energy of the GUS of Burgers' equation
are constants of motion as we will show in Theorem \ref{eliseo}.
This result holds if we work in $\mathcal{C}^{1}(I^{\ast},U_{\Lambda}(\mathbb{R}))$,
where $U_{\Lambda}\left(\mathbb{R}\right)$ is the space of ultrafunctions
described in Th. \ref{vecchiaroba}.

With this choice of the space of ultrafunctions, a GUS of the Burgers'
equation, by definition, is a solution of the following problem: 
\begin{equation}
\left\{ \begin{array}{c}
u\in\mathcal{C}^{1}(I^{\ast},U_{\Lambda}(\mathbb{R}))\ \ and\ \ \forall v\in U_{\Lambda}(\mathbb{R})\\
\\
\int\left(\partial_{t}u\right)vdx=-\int\left(u\partial_{x}u\right)vdx;\\
\\
u\left(0,x\right)=u_{0}\left(x\right),
\end{array}\right.\label{BEG}
\end{equation}
were $u_{0}\in U_{\Lambda}(\mathbb{R})$ (mostly, we will consider
the case where $u_{0}\in\left(H_{c}^{1}(\mathbb{R})\right)^{\sigma}$).
Let us recall that, by Definition \ref{def:C kappa}, for every $u\in\mathcal{C}^{1}(I^{\ast},U_{\Lambda}(\mathbb{R}))$,
we have $\partial_{t}u(t,\cdot)\in U_{\Lambda}(\mathbb{R})$.

We have the following result: 
\begin{thm}
\label{elisa}For every initial data $u_{0}\in U_{\Lambda}(\mathbb{R})$
the problem (\ref{BEG}) has a GUS. \end{thm}
\begin{proof}
It is sufficient to apply Theorem \ref{TT} to obtain the local existence
of a GUS $u$, and then Corollary \ref{coco} with 
\[
I(w)=E(w)=\frac{1}{2}\int w^{2}dx
\]
to deduce that the local GUS is, actually, global. In fact if we take
$v(t,x)=u(t,x)$ in the weak equation that defines Problem (\ref{BEG}),
we get 
\begin{eqnarray*}
\int\left(\partial_{t}u\right)udx & = & -\int\left[u\partial_{x}u\right]udx=\\
-\int\left[u\partial_{x}u\right]u\cdot\widetilde{1}dx & = & -\int_{-\beta}^{\beta}\left[u\partial_{x}u\right]udx=\\
-\frac{1}{3}\int_{-\beta}^{\beta}\partial_{x}u^{3}dx & = & 0,
\end{eqnarray*}
as $u(\beta)=u(-\beta)$. Then 
\[
\partial_{t}E(u)=0
\]
and hence Corollary \ref{coco} can be applied.\end{proof}
\begin{thm}
\label{eliseo}Problem (\ref{BEG}) has two constants of motion: the
energy 
\[
E=\frac{1}{2}\int u^{2}\ dx
\]
and the momentum 
\[
P=\int u\ dx.
\]
\end{thm}
\begin{proof}
We already proved that the energy is constant in the proof of Th.
\ref{elisa}. In order to prove that also $P$ is constant take $v=\widetilde{1}\in U_{\Lambda}(\mathbb{R})$
in equation (\ref{BEG}). Then we get 
\[
\partial_{t}P=\partial_{t}\int udx=\int\partial_{t}u\widetilde{1}dx=-\int u\partial_{x}u\widetilde{1}dx=-\frac{1}{2}\int_{-\beta}^{+\beta}\partial_{x}u^{2}dx=0,
\]
as $u(-\beta)=u(\beta)$. 
\end{proof}
Let us notice that Theorems \ref{elisa} and \ref{eliseo} hold even
if $u_{0}$ is a very singular object, e.g. a delta-like ultrafunction. 
\begin{rem}
Prop. \ref{tina} shows that the strong solutions of (\ref{BE}) have
infinitely many constants of motion; is this fact true for the GUS?
Let us try to prove that 
\[
\int G(u(t,x))dx
\]
is constant following the same proof used in Thm. \ref{elisa} and
\ref{eliseo}. We set 
\[
v(t,x)=P_{\Lambda}G^{\prime}(u)\in C(I^{\ast},U_{\Lambda})
\]
and we replace it in eq. (\ref{BEG}), so that 
\begin{eqnarray*}
\partial_{t}\int G(u(t,x))dx & = & \int\partial_{t}uG^{\prime}(u)dx\\
 & = & \int\partial_{t}uP_{\Lambda}G^{\prime}(u)dx\ \ (\text{since\ }\partial_{t}u(t,\cdot)\in U_{\Lambda})\\
 & = & -\int u\partial_{x}uP_{\Lambda}G^{\prime}(u)dx.
\end{eqnarray*}

\end{rem}
Now, if we assume that $G^{\prime}(u(t,.))\in U_{\Lambda}(\mathbb{R})$,
we have that $P_{\Lambda}G^{\prime}(u)=G^{\prime}(u)$ and hence 
\begin{eqnarray*}
\partial_{t}\int G(u(t,x))dx & = & -\int u\partial_{x}uG^{\prime}(u)dx\\
 & = & -\int\partial_{x}H(u)dx=0
\end{eqnarray*}
where $H(u)$ is defined by (\ref{lilla}). Thus $\int G(u(t,x))dx$
is a constant of motion provided that 
\begin{equation}
G^{\prime}(u)\in C(I,U_{\Lambda}).\label{romina}
\end{equation}
However, this is only a sufficient condition. Clearly, in general
the analogous of condition (\ref{romina}) will depend on the choice
of the space of ultrafunctions $V_{\Lambda}(\mathbb{R})$: different
choices of this space will give different constants of motion. Our
choice $V_{\Lambda}(\mathbb{R})=U_{\Lambda}(\Omega)$ was motivated
by the fact that GUS of equation (\ref{BEG}) in $U_{\Lambda}(\Omega)$
preserves both the energy and the momentum.

\subsection{GUS and weak solutions of BE}

In this section we consider equation (\ref{BEG}) with $u_{0}\in\left(H_{c}^{1}(\mathbb{R})\right)^{\sigma}$.
Our first result is the following: 
\begin{thm}
\label{thm:approssimazione}Let $u$ be the GUS of problem (\ref{BEG})
with initial data $u_{0}\in\left(H_{c}^{1}(\mathbb{R})\right)^{\sigma}$.
Then $[u]_{\mathscr{D}(I\times\Omega)}$ is a weak solution of problem
\ref{BE}.\end{thm}
\begin{proof}
From Theorem \ref{elisa} we know that the problem admits a GUS $u$,
and from Theorem \ref{eliseo} we deduce that $[u]_{\mathscr{D}}$
is a bounded generalized distribution: in fact, for every $\varphi\in\mathscr{D}(I\times\mathbb{R})$
we have 
\[
\left|\int\overline{u}\varphi dx\right|\leq\left(\int\overline{u}^{2}dx\right)^{\frac{1}{2}}\left(\int\varphi^{2}dx\right)^{\frac{1}{2}}<+\infty
\]
as $\int\overline{u}^{2}dx=\int u_{0}^{2}dx<+\infty$ by the conservation
of energy on GUS. Therefore, from Th.~\ref{mina} we deduce that
$w:=\left[\bar{u}\right]_{\mathscr{D}(I\times\Omega)}$ is a weak
solution of problem \ref{BE}. 
\end{proof}
Thus the GUS of problem \ref{BEG} is unique and it is associated
with a weak solution of problem \ref{BE}. It is well known (see e.g.
\cite{Bianchini} and references therein) that weak solutions of (\ref{BE})
are not unique: hence, in a certain sense, the ultrafunctions give
a way to choose a particular weak solution among the (usually infinite)
weak solutions of problem \ref{BE}.

However, among the weak solutions there is one that is of special
interest, namely the \textbf{entropy solution}. The entropy solution
is the only weak solution of (\ref{BE}) satisfying particular conditions
(the entropy conditions) along the curves of discontinuity of the
solution (see e.g. \cite{evans}, Chapter 3). For our purposes, we
are interested in the equivalent characterization of the entropy solution
as the limit, for\footnote{In this approach, $\nu$ is usally called the viscosity.}
$\nu\rightarrow0$, of the solutions of the following parabolic equations:
\begin{equation}
\frac{\partial u}{\partial t}+u\frac{\partial u}{\partial x}=\nu\frac{\partial^{2}u}{\partial x^{2}}\label{VBE}
\end{equation}
(see e.g. \cite{Hopf} for a detailed study of such equations). These
equations are called the viscous Burgers' equations and they have
smooth solutions in any reasonable function space. In particular,
in Lemma \ref{LEmmaEntropy}, we will prove that the problem \ref{VBE}
has a unique GUS in $U_{\Lambda}(\mathbb{R})$ for every initial data
$u_{0}\in U_{\Lambda}(\mathbb{R})$. Now, if $\bar{u}$ is the GUS
of problem \ref{VBE} with a classical initial condition $u_{0}\in L^{2}(\mathbb{R})$,
then $[\overline{u}]_{\mathscr{D}}$ is bounded: in fact, for every
$\varphi\in\mathscr{D}(I\times\mathbb{R})$ we have 
\[
\left|\int\overline{u}\varphi dx\right|\leq\left(\int\overline{u}^{2}dx\right)^{\frac{1}{2}}\left(\int\varphi^{2}dx\right)^{\frac{1}{2}}<+\infty
\]
as $\int\overline{u}^{2}dx\leq\int u_{0}^{2}dx<+\infty$. Therefore,
from Th. \ref{mina} we deduce that $w:=\left[\bar{u}\right]_{\mathscr{D}(I\times\Omega)}$
is a weak solution.

We are now going to prove that it is possible to choose $\nu$ infinitesimal
in such a way that $w$ is the entropy solution. This fact is interesting
since it shows that this GUS represents properly, from a Physical
point of view, the phenomenon described by Burgers' equation. In order
to see this let us consider the problem (\ref{VBE}) with $\nu$ hyperreal. 
\begin{lem}
\label{LEmmaEntropy}The problem 
\begin{equation}
\left\{ \begin{array}{c}
u\in\mathcal{C}^{1}(I,U_{\Lambda}(\Omega))\ \text{and}\ \forall v\in U_{\Lambda}(\mathbb{R})\\
\\
\int\left(\partial_{t}u(t,x)+u\partial_{x}u(t,x)\right)v(x)dx=\int\nu\partial_{x}^{2}u(t,x)v(x)dx,\\
\\
u\left(0\right)=u_{0}
\end{array}\right.\label{dora}
\end{equation}
has a unique GUS for every $\nu\in\left(\mathbb{R}^{+}\right)^{\ast}$
and every $u_{0}\in U_{\Lambda}(\mathbb{R})$. \end{lem}
\begin{proof}
Let $(U_{\lambda}(\mathbb{R}))_{\lambda\in\mathfrak{L}}$ be an approximating
net of $U_{\Lambda}(\mathbb{R})$. Since $\nu\in\left(\mathbb{R}^{+}\right)^{\ast}$
and $u_{0}\in U_{\Lambda}(\mathbb{R})$, we have that for every $\lambda\in\mathfrak{L}$
there exist $\nu_{\lambda}\in\mathbb{R}^{+}$ and $u_{0,\lambda}\in U_{\lambda}(\mathbb{R})$
such that 
\[
\nu=\lim_{\lambda\uparrow\Lambda}\nu_{\lambda}\ \text{and}\ u_{0}=\lim_{\lambda\uparrow\Lambda}u_{0,\lambda}.
\]
Thus, we can consider the approximate problems 
\begin{equation}
\left\{ \begin{array}{c}
u\in\mathcal{C}^{1}(I,U_{\lambda}(\mathbb{R}))\ and\ \ \forall v\in U_{\lambda}(\mathbb{R})\\
\\
\int\left(\partial_{t}u(t,x)+u\partial_{x}u(t,x)\right)v(x)dx=\int\nu\partial_{x}^{2}u(t,x)v(x)dx,\\
\\
u\left(0\right)=u_{0,\lambda}.
\end{array}\right.\label{jessica}
\end{equation}
For every $\lambda$, the problem (\ref{jessica}) has a unique solution
$u_{\lambda}$. If we let $u_{\Lambda}=\lim_{\lambda\uparrow\Lambda}u_{\lambda}$
we have that $u_{\Lambda}$ is the unique ultrafunction solution of
problem (\ref{dora}). 
\end{proof}
Let us call $u_{\nu}$ the GUS of Problem (\ref{dora}). A natural
conjecture would be that, if $u_{0}$ is standard, then for every
$\nu$ infinitesimal the distribution $\left[u_{\nu}\right]_{\mathscr{D}(I\times\Omega)}$
is the entropy solution of Burgers' equation. However, as we are going
to show in the following Theorem, in general this property is true
only ``when $\nu$ is a \emph{large} infinitesimal'': 
\begin{thm}
\label{EntropyGUS}Let $u_{0}$ be standard, let $z$ be the entropy
solution of Problem \ref{BE} with initial condition $u_{0}$ and,
for every $\nu\in\mathbb{R^{\ast}}$, let $u_{\nu}$ be the solution
of Problem \ref{dora} with initial condition $u_{0}^{\ast}$. Then
there exists an infinitesimal number $\nu_{0}$ such that, for every
infinitesimal $\nu\geq\nu_{0}$, $\left[u_{\nu}\right]_{\mathscr{D}(I\times\Omega)}=z$;
namely, the GUS of Problem \ref{dora}, for every infinitesimal $\nu\geq\nu_{0}$,
correspond (in the sense of Definition \ref{DEfCorrespondenceDistrUltra})
to the entropy solution of Problem \ref{BE}.\end{thm}
\begin{proof}
For every real number $\nu$ we have that the standard problem 
\[
\left\{ \begin{array}{c}
w\in\mathcal{C}^{1}(I,H_{\flat}^{1}(\mathbb{R})),\\
\\
\partial_{t}w(t,x)+w\partial_{x}w(t,x)=\nu\partial_{x}^{2}w(t,x),\\
\\
w\left(0\right)=u_{0}
\end{array}\right.
\]
has a unique solution $w_{\nu}$. Therefore for every real number
$\nu$ we have $u_{\nu}=w_{\nu}^{\ast}$. For overspill we therefore
have that there exists an infinitesimal number $\nu_{0}$ such that,
for every infinitesimal $\nu\geq\nu_{0}$, $u_{\nu}=w_{\nu}$, where
$w_{\nu}$ is the solution of the problem 
\[
\left\{ \begin{array}{c}
w\in\mathcal{C}^{1}(I,H_{\flat}^{1}(\mathbb{R}))^{\ast},\\
\\
\partial_{t}w(t,x)+w\partial_{x}w(t,x)=\nu\partial_{x}^{2}w(t,x),\\
\\
w\left(0\right)=u_{0}^{\ast}.
\end{array}\right.
\]

But as $z=\lim\limits _{\varepsilon\rightarrow0^{+}}v_{\varepsilon}$,
we have that for every infinitesimal number $\nu$, for every test
function $\varphi$ we have that 
\[
\left\langle z^{\ast}-v_{\nu},\varphi^{\ast}\right\rangle \sim0.
\]

In particular for every infinitesimal $\nu\geq\nu_{0}$, 
\[
\left\langle z^{\ast}-u_{\nu},\varphi^{\ast}\right\rangle \sim0,
\]
and as this holds for every test function $\varphi$ we have our thesis. 
\end{proof}
Theorem \ref{EntropyGUS} shows that, for a standard initial value
$u_{0}$, there exists a ultrafunction which corresponds to the entropy
solution of Burgers' equation; moreover, this ultrafunction solves
a viscous Burgers' equation for an infinitesimal viscosity (namely,
it is the solution of an infinitesimal perturbation of Burgers' equation).
However, within ultrafunctions theory there is another ``natural''
solution of Burgers' equation for a standard initial value $u_{0}$,
namely the unique ultrafunction $u$ that solves Problem \ref{BEG}.
We already proved in Theorem \ref{thm:approssimazione} that $u$
corresponds (in the sense of Definition \ref{DEfCorrespondenceDistrUltra})
to a weak solution of Burgers' equation. Our conjecture is that this
weak solution is precisely the entropy solution; however, we have
not been able to prove this (yet!). Nevertheless, in any case it makes
sense to analyse this solution: this will be done in the next section.

\subsection{The microscopic part\label{sub:The-microscopic-part}}

Let $u\in C^{1}(I^{\ast},U_{\Lambda})$ be the GUS of (\ref{BEG})
and let $w=[u]_{\mathscr{D}}$. With some abuse of notation we will
identify the distribution $w$ with a $L^{2}$ function. We want to
compare $u$ and $w^{*}$ and to give a physical interpretation of
their difference.

Since we have that

\[
\left[u\right]_{\mathscr{D}}=\left[w^{*}\right]_{\mathscr{D}}
\]
we can write 
\[
u=w^{\ast}+\psi;
\]
we have that 
\[
\forall\varphi\in\mathscr{D}\left(I\times\Omega\right),\int\int^{\ast}u\varphi^{\ast}\ dx\ dt\sim\int\int w\varphi\ dx\ dt
\]
and 
\begin{equation}
\int\int^{\ast}\psi\varphi^{\ast}\ dx\ dt\sim0.\label{micra}
\end{equation}
We will call $w$ (and $w^{\ast}$) the macroscopic part of $u$ and
$\psi$ the microscopic part of $u$; in fact, we can interpret (\ref{micra})
by saying that $\psi$ does not appear to a mascroscopic analysis.
On the other hand, $\int^{\ast}\psi\varphi\ dx\ dt\not\sim0$ for
some $\varphi\in C^{1}(I^{\ast},U_{\Lambda})\backslash\mathscr{D}\left(I\times\Omega\right)$.
Such a $\varphi$ ``is able'' to detect the \emph{infinitesimal
oscillations} of $\psi$. This justifies the expression \textquotedbl{}macroscopic
part\textquotedbl{} and \textquotedbl{}microscopic part\textquotedbl{}.
So, in the case of Burgers equation, the ultrafunctions do not produce
a solution to a problem without solutions (as in the example of section
\ref{se}), but they give a different description of the phenomenon,
namely they provide also the information contained in the microscopic
part $\psi.$

So let us analyze it: 
\begin{prop}
The microscopic part $\psi$ of the GUS solution of problem (\ref{BEG})
satisfies the following properties: 
\begin{enumerate}
\item \label{enu:the-momentum-of}the momentum of $\psi$ vanishes: 
\[
\int\psi\ dx=0;
\]

\item \label{enu:-and-}$w^{\ast}$ and $\psi$ are almost orthogonal: 
\[
\int\int\psi w^{\ast}\ dxdt\sim0;
\]

\item \label{enu:the-energy-of}the energy of $u$ is the sum of the kinetic
macroscopic energy, $\int\left\vert w(t,x)\right\vert ^{2}\ dx$,
the kinetic microscopic energy (heat)\textup{ $\int\left\vert \psi(t,x)\right\vert ^{2}\ dx$
}\textup{\emph{and an infinitesimal quantity;}} 
\item \label{enu:if--is}if $w$ is the entropy solution then the ``heat''
$\int\left\vert \psi(t,x)\right\vert ^{2}\ dx$ increases. 
\end{enumerate}
\end{prop}
\begin{proof}
\ref{enu:the-momentum-of}) $\int\psi\,dx=\int u\,dx-\int w^{\ast}\,dx$,
and the conclusion follows as both $u$ and $w$ preserve the momentum.

\ref{enu:-and-}) First of all we observe that the $L^{2}$ norm of
$\psi$ is finite, as $\psi=u-w^{\ast}$ and the $L^{2}$ norms of
$u$ and $w$ are finite. Now let $\{\varphi_{n}\}_{n\in\mathbb{N}}$
be a sequence in $\mathscr{D}(I\times\Omega)$ that converges strongly
to $w$ in $L^{2}$. Let $\{\varphi_{\nu}\}_{\nu\in\mathbb{N}^{\ast}}$
be the extension of this sequence. As $\varphi_{n}\rightharpoonup w$
in $L^{2}$, we have that for any infinite number $N\in\mathbb{N}$
$\left\Vert \varphi_{N}-w^{\ast}\right\Vert _{L^{2}}\sim0$. For every
finite number $n\in\mathbb{N^{\ast}}$ we have that 
\[
\int\psi\varphi_{n}dxdt=0,
\]
as $\varphi_{n}\in\mathscr{D}(I\times\Omega$). By overspill, there
exists an infinite number $N$ such that $\int\psi\varphi_{N}dxdt=0.$
If we set $\eta=w^{\ast}-\varphi_{N}$, we have $\left\Vert \eta\right\Vert _{L^{2}}\sim0$.
Then 
\begin{eqnarray*}
\left|\int\psi w^{\ast}dxdt\right| & = & \left|\int\psi(\varphi_{N}+\eta)dxdt\right|\\
 & = & \left|\int\psi\varphi_{N}dxdt+\int\psi\eta dxdt\right|\sim0,
\end{eqnarray*}
as $\int\psi\varphi_{N}dxdt=0$ and $\left|\int\psi\eta dxdt\right|\leq\int\left|\psi\right|\left|\eta\right|dxdt\leq\left(\left\Vert \psi\right\Vert _{L^{2}}\cdot\left\Vert \eta\right\Vert _{L^{2}}\right)^{\frac{1}{2}}\sim0$.

\ref{enu:the-energy-of}) This follows easily from (\ref{enu:-and-}).

\ref{enu:if--is}) The energy of $u=w^{\ast}+\psi$ is constant, while
the energy of $w^{\ast}$, if $w$ is the entropy solution, decreases.
Therefore we deduce our thesis from (\ref{enu:the-energy-of}). 
\end{proof}
Now let $\Omega\subset I\times\mathbb{R}$ be the region where $w$
is regular (say $H^{1}$) and let $\Sigma=\left(I\times\mathbb{R}\right)\backslash\Omega$
be the singular region. We have the following result: 
\begin{thm}
$\psi$ satisfies the following equation in the sense of ultrafunctions:
\[
\partial_{t}\psi+\partial_{x}\left(\mathbf{V}\psi\right)=F,
\]
where 
\begin{equation}
\mathbf{V}=\mathbf{V}(w,\psi)=w(t,x)+\frac{1}{2}\psi(t,x)\label{eq:bella}
\end{equation}
and 
\[
\mathfrak{supp}\left(F(t,x)\right)\subset N_{\varepsilon}(\Sigma),
\]
where $N_{\varepsilon}(\Sigma)$ is an infinitesimal neighborhood
of $\Sigma^{\ast}.$\end{thm}
\begin{proof}
In $\Omega$ we have that 
\[
\partial_{t}w+w\partial_{x}w=0
\]
Since $u=w+\psi$ satisfies the following equation (in the sense of
ultrafunctions), 
\[
D_{t}u+P\left(u{}_{x}\partial u\right)=0
\]
we have that $\psi$ satisfies the equation, 
\[
D_{t}\psi+P\left[\partial_{x}\left(w\psi+\frac{1}{2}\psi^{2}\right)\right]=0
\]
in $\Omega^{\ast}\backslash N_{\varepsilon}(\Sigma)$ where $N_{\varepsilon}(\Sigma)$
is an infinitesimal neighborhood of $\Sigma^{\ast}.$ 
\end{proof}
As we have seen $\psi^{2}$ can be interpreted as the density of heat.
Then $\mathbf{V}$ can be interpreted as the flow of $\psi$; it consists
of two parts: $w$ which is the macroscopic component of the flow
and $\frac{1}{2}\psi(t,x)$ which is the transport due to the Brownian
motion.


\begin{thebibliography}{10}
\bibitem{AmMa2003} Ambrosetti A., Malchiodi A., \textsl{Nonlinear
Analysis and Semilinear Elliptic Problems, }\textsl{\emph{Cambridge
University Press}}\emph{,} Cambridge (2007).

\bibitem{benci99} Benci V., \textsl{An algebraic approach to nonstandard
analysis, }in: Calculus of Variations and Partial differential equations,
(G.Buttazzo, et al., eds.), Springer, Berlin (1999), 285-326.

\bibitem{milan} Benci V., \textsl{Hylomorphic solitons,}\textit{\ }Milan
J. Math.\textbf{ }77,\textbf{ }(2009), 271-332.

\bibitem{ultra} Benci V., \textsl{Ultrafunctions and generalized
solutions,} in: Adv. Nonlinear Stud. 13, (2013), 461--486, arXiv:1206.2257.

\bibitem{BDN200301} Benci V., Di Nasso M., \textit{A ring homomorphism
is enough to get nonstandard analysis}, Bulletin of the Belgian Mathematical
Society, vol. 10 (2003), 1-10.

\bibitem{BDN2003} Benci V., Di Nasso M., \textsl{Alpha-theory: an
elementary axiomatic for nonstandard analysis}, Expo. Math. 21, (2003),
355-386.

\bibitem{belu2012} Benci V., Luperi Baglini L., \textsl{A model problem
for ultrafunctions}$,$ in: Variational and Topological Methods: Theory,
Applications, Numerical Simulations, and Open Problems. Electron.
J. Diff. Eqns., Conference 21 (2014), 11-21.

\bibitem{belu2013} Benci V., Luperi Baglini L., \textsl{Basic Properties
of ultrafunctions,} to appear in the WNDE2012 Conference Proceedings,
arXiv:1302.7156.

\bibitem{milano} Benci V., Luperi Baglini L., \textsl{Ultrafunctions
and applications}, to appear on DCDS-S (Vol. 7, No. 4) August 2014,
arXiv:1405.4152.

\bibitem{algebra} Benci V., Luperi Baglini L., \textsl{A non archimedean
algebra and the Schwartz impossibility theorem,}, Monatsh. Math. (2014),
DOI 10.1007/s00605-014-0647-x.

\bibitem{beyond} Benci V., Luperi Baglini L., \textsl{Generalized
functions beyond distributions}, AJOM (2014), arXiv:1401.5270.

\bibitem{gauss} Benci V., Luperi Baglini L., \emph{A generalization
of Gauss' divergence theorem}, Contemporary mathematics, to be published.

\bibitem{topologia} Benci V., Luperi Baglini L., \textsl{A topological
approach to non-Archimedean mathematics}, submitted, arXiv: 1412.2223.

\bibitem{befolib} Benci V. Fortunato D.\textit{, Vational methods
in nonlinear field equations, }Springer Monographs in Mathematics,
to appear.

\bibitem{biagioni} Biagioni H.~A., \textit{Generalized solutions
to nonlinear first-order systems}, Monatsh. Math 118 (1994), 7-20.

\bibitem{biagioni-MO} Biagioni H.~A., Oberguggenberger M., \textit{Generalized
solutions to Burgers' equation}, J. Differ. Equations 97 (1992), 263-287.

\bibitem{Bianchini} Bianchini S., Marconi E., \emph{On the concentration
of entropy for scalar conservation laws}, DCDS-S 9 (2016), 73-88,
doi:10.3934/dcdss.2016.9.73.

\bibitem{Ca03} Cazenave T., \textsl{Semilinear Schrödinger equations},
Courant Lecture Notes in Mathematics, vol.~10, New York University
Courant Institute of Mathematical Sciences, New York, 2003.

\bibitem{MinEnt} De Lellis C., Otto F., Westdickenberg M., \textit{Minimal
entropy conditions for Burgers equation}, Quart. Appl. Math. 62 (2004),
687-700.

\bibitem{evans} Evans C., \textit{Partial differential equations,
}Graduate Studies in Mathematics, Vol. 19, AMS, Providence, 1991.\textit{ }

\bibitem{Hopf} Hopf E., \textit{The partial differential equation
$u_{t}+uu_{x}=\mu u_{xx}$}, Comm. Pure Appl. Math. 3 (1950), 201-230.

\bibitem{keisler76} Keisler H.~J., F\textsl{oundations of Infinitesimal
Calculus}, Prindle, Weber \& Schmidt, Boston, (1976).

\bibitem{JLL} Lions J.~L.,\textsl{ Quelques methodes de resolution
des problemes aux limites non lineaires}, Paris, Dunod, (1969).

\bibitem{rob} Robinson A., N\textsl{on-standard Analysis,}Proceedings
of the Royal Academy of Sciences, Amsterdam (Series A) 64, (1961),
432-440.

\bibitem{ros} Rosinger E.~E., \textit{Generalized solutions of nonlinear
partial differential equations}, in ``North-Holland Math. Studies,\textquotedblright{}
Vol. 146, North-Holland, Amsterdam, 1987.\end{thebibliography}
\end{document}